\documentclass[11pt]{amsart}

\calclayout

\usepackage[T1]{fontenc}
\usepackage[usenames]{color}
\usepackage{amsmath,amssymb,amsthm}
\usepackage{wrapfig}
\usepackage{tikz-cd}
\usepackage{verbatimbox}

\usepackage{enumerate}

\usepackage{caption}
\usepackage{mathrsfs}

\definecolor{PineGreen}{rgb}{0.0,0.47,0.44}
\definecolor{MidnightBlue}{rgb}{0.1,0.1,0.44}
\definecolor{magenta}{rgb}{1.0,0.0,1.0}
\definecolor{bl1}{HTML}{4479A1}
\definecolor{pur1}{HTML}{52196D}
\definecolor{mag1}{HTML}{2AD0F1}
\definecolor{org1}{rgb}{.92,.39.21}
\definecolor{pur2}{rgb}{.53,.47,.7}

\usepackage{hyperref}
\hypersetup{
  colorlinks=true,
  linkcolor=MidnightBlue,
  citecolor=PineGreen,
  filecolor=magenta,
  urlcolor=MidnightBlue
}

\makeatletter
\newcommand{\eqnum}{\refstepcounter{equation}\textup{\tagform@{\theequation}}}
\makeatother

\newtheorem{theorem}{Theorem}
\numberwithin{theorem}{section}
\newtheorem{proposition}[theorem]{Proposition}
\newtheorem{lemma}[theorem]{Lemma}
\newtheorem{corollary}[theorem]{Corollary}
\theoremstyle{definition}

\newtheorem{definition}[theorem]{Definition}
\newtheorem{convention}[theorem]{Convention}
\theoremstyle{remark}
\newtheorem{remark}[theorem]{Remark}
\newtheorem{example}[theorem]{Example}

\newcommand{\bbA}{\mathbb{A}}

\newcommand{\PP}{\mathbb{P}}
\newcommand{\pp}{\mathbb{P}}
\newcommand{\CC}{\mathbb{C}}
\newcommand{\ZZ}{\mathbb{Z}}

\newcommand{\pr}{pr}

\DeclareMathOperator{\codim}{codim}
\DeclareMathOperator{\segre}{s}

\begin{document}

\title{Segre class computation and practical applications}
\author{ Corey Harris }
\author{ Martin Helmer }

\begin{abstract} 
Let $X \subset Y$ be closed (possibly singular) subschemes of a smooth projective toric variety $T$.
We show how to compute the Segre class $\segre(X,Y)$ as a class in the Chow group of $T$.
Building on this, we give effective methods to compute intersection products in projective varieties, to determine algebraic multiplicity \textit{without} working in local rings, and to test pairwise containment of subvarieties of $T$. Our methods may be implemented without using Gr\"obner bases; in particular any algorithm to compute the number of solutions of a zero-dimensional polynomial system may be used. 
\end{abstract}

\maketitle\vspace{-0.05in}
\section{Introduction}
Segre classes capture important enumerative and geometric properties of systems of polynomial equations coming from embeddings of schemes. Historically, these classes have played a fundamental role in the development of Fulton-MacPherson intersection theory \cite[\S6.1]{fulton2013intersection}. Computation of Segre classes (other than in a few special cases) has proven to be a challenge; this has limited the development of applications in practice. 

Evidence of the significance of Segre classes in algebraic geometry can be found in the fact that many important characteristic classes can be written as $C \frown s(X,Y)$, where $C$ is some polynomial in the Chern classes of vector bundles on $X$.  The flagship example of this is the topological Euler characteristic $\chi(X)$, which appears in the Chern-Schwartz-MacPherson class $c_{SM}(X)$. By results of Aluffi \cite{Aluffi2003c,aluffi2018chern} the class $c_{SM}(X)$ can be directly obtained by computing a Segre class. There are also formulas in terms of Segre classes for the Milnor class of a hypersurface \cite{Aluffi2003c}, the Chern-Mather class and polar degrees \cite{Piene1978}, and the Euclidean distance degree of a projective variety \cite{AH17}.

More generally, many enumerative problems end up in the situation of an \emph{excess intersection}, 
in which an intersection is expected to be finite but instead is the union of a finite set of points along with a positive-dimensional set.  Typically the desired quantity is then the number of points outside the positive-dimensional part. The Segre class gives a way to express the \emph{contribution} of this part, which is the difference between the expected number (e.g., the B\'ezout bound) and the actual number of points in the finite set.

Let $X \subset Y$ be closed subschemes of a smooth projective toric variety $T$. 
The Segre class of $X$ in $Y$ is a class $s(X,Y) \in A_*(X)$ in the Chow group of $X$. Since the group $A_*(X)$ is often unknown, the best we could hope for in general is to compute the pushforward of this class to $A_*(T)$.

Previous work on computing Segre classes of the form $s(X,\mathbb{P}^n)$ in $A_*(\pp^n)$ (i.e.,~the special case where $Y=T=\pp^n$) began with the paper \cite{Aluffi2003c} and 
alternative methods were developed in \cite{Eklund2013,helmer2016proj}. These methods were generalized to compute $s(X,T)$ in $A_*(T)$ in \cite{Moe2013,helmer2017toric}.  
In \cite{Harris2017} the scope was extended to compute the Segre class $s(X,Y)$ pushed forward to $A_*(\pp^n)$ for $X \subset Y \subset \pp^n$ and $Y$ a variety. The present work goes further by taking arbitrary subschemes not just of projective space, but of any smooth projective toric variety, to obtain $s(X,Y)$ in $A_*(T)$. 

The ability to effectively compute Segre classes in this new setting opens the way for several novel computational applications. For instance, computing Segre classes in products $T \times T$ allows for a general framework to compute intersection products in subvarieties of $T$.
While a big part of this paper is devoted to studying Segre classes, many of the resulting applications can be expressed without them.  
In particular, we show that algebraic multiplicity can be computed and pairwise containment of varieties can be tested by counting the number of points in a single zero-dimensional set.

\subsection{Examples}We now give three examples which illustrate how the results developed in this paper give rise to new methods to answer classical geometric and enumerative questions. The first example shows how the intersection product  can be computed using a Segre class. Following this we give examples which compute algebraic multiplicity and test ideal containment. The methods presented in the later two examples build on ideas developed to compute Segre classes (which we present in \S\ref{sec:MainResults}), but can be understood without them. 

\subsubsection{Intersection theory}
The \textit{intersection product} of varieties $X$ and $V$ in a non-singular variety $Y$, denoted $X\cdot_Y V$ (see \S\ref{sec:intersection-theory} for a definition), captures the behavior of the intersection $X\cap V$ inside of $Y$. 
It is a class in the Chow ring $A^*(Y)$. 
If $X$ meets $V$ transversely in the expected dimension, the intersection product may be defined by
\[X \cdot_Y V = [X \cap V] \in A^*(Y).  \]
If $X$ and $V$ do not meet dimensionally transversely, 
but there exist $X'$ and $V'$ which are respectively equivalent in $A^*(Y)$ and are dimensionally transverse,
then we define the intersection analogously: $ X \cdot_Y V = [X' \cap V'].$\\ \vspace{-1.4em}
\begin{wrapfigure}{r}{0.34\textwidth}
\vspace{0.7em}
\begin{center}
\includegraphics[scale=0.16]{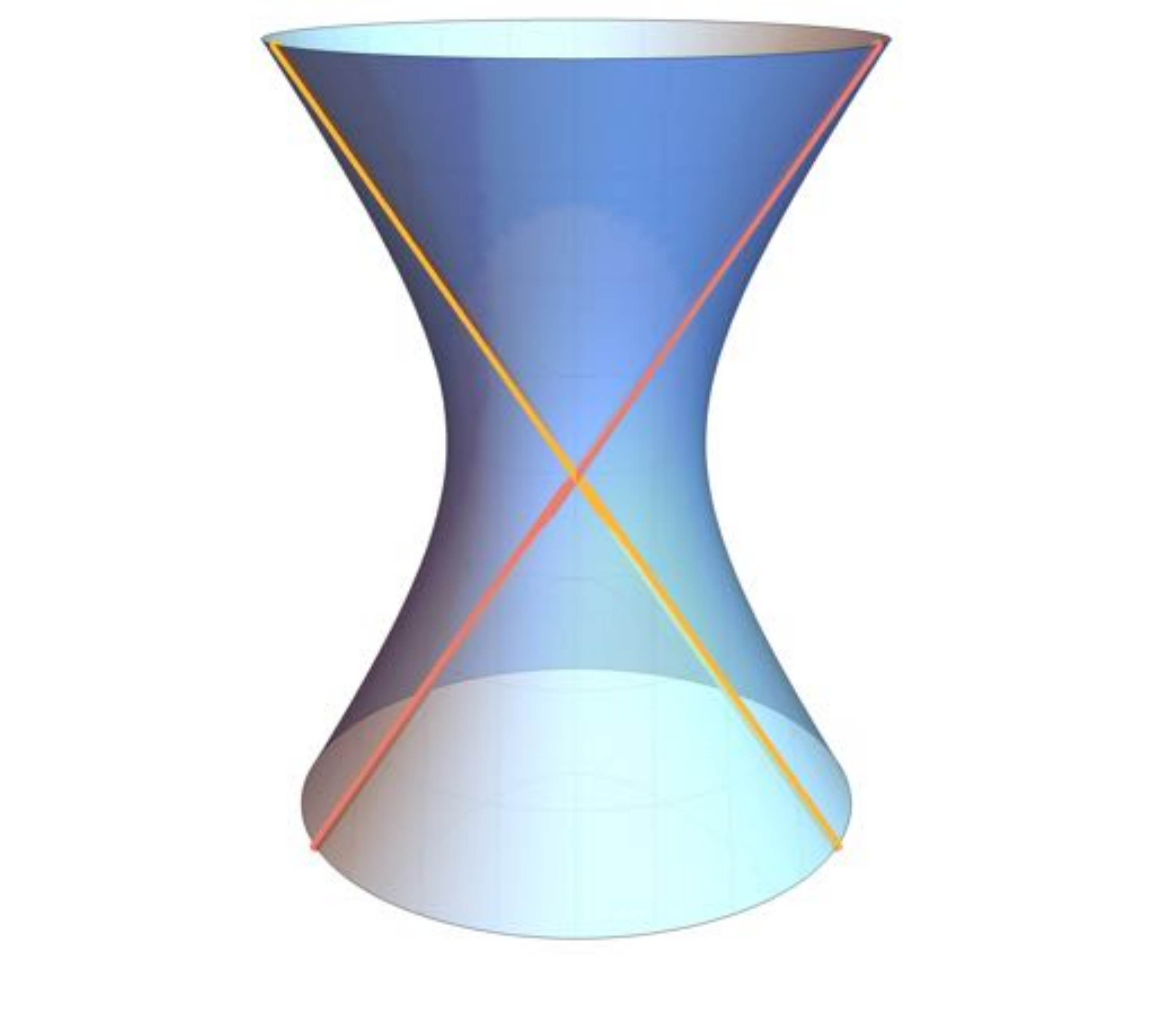}
\end{center}\caption{Lines on a quadric.\label{fig:quadric}}
\vspace{-1.5em}
\end{wrapfigure} 
\begin{example}[Lines on a quadric, I]\label{ex:quadricNaive}
The quadric surface $Q=\pp^1{\times}\pp^1 \subset \pp^3$ in Figure \ref{fig:quadric}, comes with two families of lines, $L_a = \{a\} \times \pp^1$ and $K_b = \pp^1 \times \{b\}$, and any two lines in the same family are equal in $A^*(Q)$. 
Let $L_a$ and $K_b$ be two such lines on $Q$.
We compute $L_a \cdot_Q K_b$ and $L_a \cdot_Q L_a$.

Since $L_a$ meets $K_b$ only at $(a,b)$ and the intersection is transverse of the expected dimension $\dim(L_a) - \left(\dim(Q) - \dim(K_b) \right) = 1 - (2-1) = 0$, we can conclude that $L_a \cdot_Q K_b = [pt]$ is the class of a point.  

In contrast, $L_a \cap L_a = L_a$ does not have the expected dimension, 
so we can try to ``move'' one of the terms by finding a suitable replacement with the same class in $A^*(Q)$.   
If $L_c$ is any other line, then $L_a \cap L_c$ is empty, and thus the intersection product is 0.
\end{example}
Another way to perform the computation above would be to recognize that $A^*(Q) \cong \ZZ[l,k]/\langle l^2,k^2 \rangle$ and $l = [L_a]$, $k = [K_b]$.  Then the product in this ring is the intersection product, i.e.~$l \cdot k = [pt] = L_a \cdot_Q K_b$ and $l^2 = 0 = L_a \cdot_Q L_a$.  However, when computing in practice, one may encounter a variety for which the Chow ring is not known a priori.  In this situation, using the methods of \S\ref{sec:MainResults}-\ref{sec:intersection-theory}, we could still arrive at the answer.
\begin{example}[Lines on a Quadric, II] \label{ex:quadricSegre}
Let $Q = \mathbb{V}(xy-zw) \subset \pp^3$ and forget that $Q$ is a smooth quadric.  
Given subvarieties $L$ and $K$ defined on $Q$ by $\{x=0=w\}$ and $\{y=0=w\}$, respectively, we wish to compute the intersection products $L \cdot_Q K$ and $L \cdot_Q L$.  Since we do not know the Chow ring, we must do something new. Specifically we work in the Chow ring of the ambient $\pp^3$, which is $A^*(\pp^3) \cong \ZZ[h]/\langle h^4\rangle$, and write the intersection product in terms of classes there.

In \S\ref{sec:intersection-theory}, we will see that a universal setup for this approach is to intersect $L \times K$ with the diagonal in $Q \times Q$.  Then the ingredients to our computation are:
\begin{enumerate}[(i)]
\item the \emph{Chern class} $i_* c(TQ) =  1 + 2h + 2h^2$, where $i:Q \hookrightarrow \pp^3$ is the inclusion,
\item the \emph{Segre class} $\Delta^* s(L \cap K, Q \times Q) = \Delta^*(h_1^3h_2^3) =  h^3$,
\item the \emph{Segre class} $\Delta^* s(L \cap L, Q \times Q) = \Delta^*(h_1^3h_2^2+h_1^2h_2^3-2h_1^3h_2^3) = h^2 - 2h^3$,
\end{enumerate}
where $\Delta: \pp^3 \to \pp^3 \times \pp^3$ is the diagonal map and $h_1,h_2$ are the hyperplane classes for each factor.
As above, we will frequently suppress obvious pushforwards.
Then, as follows from Theorem \ref{thm:intersectionProduct}, we confirm $L \cdot_Q K = \{(1 + 2h + 2h^2)  h^3\}_{exp.~dim.} = h^3 \in A^*(\pp^3)$ and $L \cdot_Q L =\{(1 + 2h + 2h^2) (h^2 - 2h^3)\}_{exp.~dim.} = 0$.  
\end{example}
\begin{remark} In Example \ref{ex:quadricSegre} we did not need to do any computations by hand.  Items (i)-(iii) can all be computed in a computer algebra system given the defining equations.  
Following \cite{Aluffi2003c}, 
$i_*c(TQ)$ is determined also by a Segre class (specifically the Segre class $i_*s(Q,\pp^3)$), which could already be found via the methods of \cite{Aluffi2003c,Eklund2013,helmer2016proj},
and computing items (ii) and (iii) are contributions of this paper.
\end{remark}

\subsubsection{Algebraic multiplicity}
Our second example illustrates how Segre classes can be used to compute the algebraic multiplicity of a local ring with respect to an ideal (without computing in the local ring). 

\begin{example}[Algebraic multiplicity along a component]\label{ex:algMult}
Let $R=\CC[x,y,z,w]$ be the homogeneous coordinate ring of $\PP^3$. 
As in \cite[Ex.~III.10]{sayrafi2017computations}, we consider the twisted cubic $X \subset \pp^3$ defined by the prime ideal
\[\mathcal{I}_X = \langle y w-z^{2},xw-y z,xz-y^{2} \rangle\] 
and the scheme $Y$ defined by the ideal 
\[\mathcal{I}_Y =\langle z(yw-z^2)-w(xw-yz), xz-y^2 \rangle.\] 
Then $X$ is a subscheme of $Y$, since $\mathcal{I}_Y \subset \mathcal{I}_X$. 
Let $\mathcal{O}_{X,Y}=(R/{\mathcal{I}_Y})_{\mathcal{I}_X}$ be the local ring of $Y$ along $X$. The \textit{algebraic multiplicity} of $Y$ along $X$ is the leading coefficient $e_XY$ of the Hilbert-Samuel polynomial associated to the local ring $ \mathcal{O}_{X,Y}$ (see \S\ref{subSection:AlgMult1}).
The multiplicity may be read off of the Segre class $s(X,Y)$ since $e_XY$ is also the coefficient of $[X]$ in the class $s(X,Y) \in A^*(X)$, see Definition \ref{def:eXY}. 
The class $[X]$ is $3h^2 \in A^*(\pp^3)$. Applying the methods of \S\ref{sec:MainResults} we find that
\[
s(X,Y)=6 h^2=2 (3h^2) \in A^*(\PP^3).
\]
Together this gives that {$e_XY=2$}. Let $d$ be the maximum degree among the defining equations of the ideals $\mathcal{I}_X$ and $\mathcal{I}_Y$ and let $g$ be the \emph{dimension-$X$ projective degree} defined by $X$ in $Y$ (see \S\ref{subsec:ProjectiveDegrees}). More directly, by Theorem \ref{theorem:eXY_PPn} we have that 
\[
e_XY=\frac{\deg(Y)d^{\dim(Y)-\dim(X)}-g}{\deg(X)}=\frac{6\cdot 3^0 -0}{3}=2.
\]
\end{example}
\subsubsection{Containment}Our third example demonstrates a new criterion to test containment of varieties
or of irreducible components of schemes (see \S\ref{sec:GBFreeContainment}).
\begin{example}[Containment of Varieties]
Work in $\pp^6$ with coordinates $x_0,\dots,x_6$ and let
\[
  I=\text{the ideal defined by all } 3{\times}3 \text{ minors of}
  \begin{pmatrix}
    {x}_{0}&
      {x}_{1}&
      {x}_{2}&
      {x}_{3}\\
      3 {x}_{3}&
      4 {x}_{4}&
      5 {x}_{5}&
      6 {x}_{6}\\
      {x}_{2}&
      {x}_{3}&
      {x}_{4}&
      {x}_{5}\\
      {x}_{0}+5 {x}_{1}&
      {x}_{1}+6 {x}_{2}&
      {x}_{2}+7 {x}_{3}&
      {x}_{3}+8 {x}_{4}\\
    \end{pmatrix}.
  \]
  The variety $Y=\mathbb{V}(I)$ is an irreducible singular surface of degree 20 in $\pp^6$.
  
Now set $K=\langle x_0,x_1,x_2,x_3,x_4 \rangle$ and $X=\mathbb{V}(K)$. We seek to determine if the line $X$ is contained in the singular locus ${\rm Sing}(Y)$ of the surface $Y$. The standard method to test for this containment is to compute the ideal $J$ defining ${\rm Sing}(Y)$ and reduce each generator of $J$ with respect to a Gr\"obner basis for $K$. In this case the ideal $J$ is defined by the $4\times 4$ minors of the Jacobian matrix of $I$; it is clear from the structure of $I$ (it has $16$ generators of degree three in six variables) that the computation of the minors to obtain the ideal $J$ will be very time consuming.

On the other hand, by Corollary \ref{cor:contSingLocusPn} we have that $X\subset {\rm Sing}(Y)$ if and only if
\[ 
  \frac{20\cdot 3^{2-1}-g}{1}>1
\]
where the left-hand side is $e_XY$ (computed via Theorem \ref{theorem:eXY_PPn}), and $g$ is the {dimension-$X$ projective degree} of $X$ in $Y$ (see \S\ref{subsec:ProjectiveDegrees}). Using Theorem \ref{MainTheorem1_Multi_proj} we compute $g=58$ by finding the number of solutions to a single zero-dimensional system of polynomials, with each polynomial of degree at most three. Substituting this in we have that $e_XY=2>1$ and, hence, $X\subset {\rm Sing}(Y)$. The computation of the integer $g=58$ takes approximately $0.07$ seconds using Macaulay2 \cite{M2} on a laptop. 
      
We note again that the test described above using Corollary \ref{cor:contSingLocusPn} does \textit{not} compute the ideal defining ${\rm Sing}(Y)$. In this case computing the ideal $J$ and using  Gr\"obner basis methods to test if $J\subset K $ takes approximately 692 seconds using Macaulay2 \cite{M2} on the same test machine (the majority of this time, about 690 seconds, is spent computing the ideal $J$). 
\end{example}

This paper is organized as follows. In Section \ref{section:background} we establish our notation and conventions and present relevant background on Segre classes and projective degrees. The main results of this paper are presented in Section \ref{sec:MainResults}. In \S\ref{subSec:SegreMultiProj} we consider the case where $X$ is a subscheme of an irreducible scheme $Y\subset \pp^{n_1}\times \cdots \times \pp^{n_m}$. In Theorem \ref{MainTheorem2MultiProj} we give an explicit formula for the Segre class $s(X,Y)$ in terms of the projective degrees of a rational map defined by $X$ from $Y$ to a projective space (see \S\ref{subsec:ProjectiveDegrees}). In Theorem \ref{MainTheorem1_Multi_proj} we give an expression for these projective degrees as a vector space dimension of a ring modulo a certain zero-dimensional ideal. These results are generalized in \S\ref{subsec:SegSubToric} to the case where $Y$ is a subscheme of a smooth projective toric variety $T$. 

In Section \ref{sec:intersection-theory} we show how the results of Section \ref{sec:MainResults} can be applied to compute intersection products.  If $Y \subset \pp^n$ is a smooth variety, $X \cdot_Y V$ requires computing $s(X \cap V, Y \times Y)$. The main result is that the pushforward of this class to $\pp^n$ is enough to recover the pushforward of the intersection product.

In Section \ref{subSection:AlgMult1} we use the results of \S\ref{subSec:SegreMultiProj} to give an explicit expression for $e_XY$, the algebraic multiplicity of $Y$ along $X$, in terms of ideals in the coordinate ring of $ \pp^{n_1}\times \cdots \times \pp^{n_m}$. The expressions are also generalized to $Y\subset T$. 

Finally in Section \ref{sec:GBFreeContainment} we combine the results from
\S\ref{subSection:AlgMult1} with a classical result of Samuel
\cite{samuel1955methodes} to yield new numerical tests for the containment of one
variety in another. Let $X$ and $Y$ be arbitrary subvarieties of a smooth projective toric
variety $T$. In \S\ref{subsec:subVarInSingLocus} we give a simple criterion
to determine if $X$ is contained in the singular locus of $Y$ \textit{without}
computing the defining equations of the singular locus. 
In \S\ref{subsec:ContainmentTest} we give a criterion to determine if
$X\subset Y$. 
As in the previous results, computing a Gr\"obner basis is not required and methods
from numerical algebraic geometry could be used. To the best of our knowledge, this is the
first general purpose method which is able to test containment of possibly singular varieties using only numeric methods. 

 As of version 1.13, Macaulay2 \cite{M2} contains the {\tt SegreClasses} package, which implements many of the results described in this paper.

\section{Background}\label{section:background}In this section we review several definitions and explicitly state the notations and conventions we will use throughout the paper. 
\subsection{Notations and conventions}\label{subsec:notationConvention}
We work throughout over an algebraically closed field $k=\bar k$ of characteristic
zero. 
\subsubsection{Varieties, schemes, and irreducible components}
Since we always work in an ambient projective variety, all schemes will be of finite type over the base field.
By \emph{variety} we mean a reduced and irreducible, separated scheme, that is,
an irreducible algebraic set.
Given polynomials $f_1,\dots,f_r$ we let $\mathbb{V}(f_1,\dots,f_r)$ denote the
algebraic set defined by $f_1=\cdots=f_r=0$. 
Conversely, if $Y$ is a subscheme of a smooth variety $Z$ with coordinate ring
$R$, we will let $\mathcal{I}_Y$ be the ideal in $R$ defining the scheme $Y$ and let $\sqrt{\mathcal{I}_Y}$ be the radical
ideal in $R$ defining the reduced scheme $Y_{\mathrm{red}}$.
If $Y$ is a subscheme of a smooth variety defined by an ideal $\mathcal{I}_Y$,
its 
\emph{primary components} are the schemes associated to the primary components of $\mathcal{I}_Y$;
its \textit{irreducible components} are the varieties defined
by the associated primes of $\mathcal{I}_Y$.

\subsubsection{Chow classes} 
Let $Y$ be a subscheme of a smooth variety $Z$. The irreducible components $Y_i$ of $Y$ have associated geometric multiplicity $\mathfrak{m}_i$ given by the length of the local ring $\mathcal{O}_{Y_i,Y}$, and we write $[Y] = \sum_{i=1}^t \mathfrak{m}_i [Y_i] $ for the rational equivalence class of $Y$ in $A_*(Y)$.  We frequently write $[Y] \in A_*(Z)$ to mean the pushforward via inclusion.
For a cycle class $\beta$ in
the Chow group $A_*(Z)$ we will use the notation $\int \beta$ to denote the degree of the zero-dimensional
part of $\beta$ (as in Definition 1.4 of Fulton \cite{fulton2013intersection}).
The degree of a zero-dimensional scheme $W$ is $\deg(W) = \int [W]$.  
For instance, if $W \subset \pp^2$ is defined by  $\langle x^2,y^2 \rangle$  then $\deg(W)=4$. 

\subsubsection{The total coordinate ring of a toric
  variety}\label{subsubsec:CoxRing} Let $T_{\Sigma}$ be a smooth projective
toric variety defined by a fan $\Sigma$ and let $\Sigma(1)$ denote the rays in the fan.
The \textit{Cox ring} of $T_{\Sigma}$ is $R=k[x_{\rho} \; | \; \rho \in \Sigma(1) ] $.
The ring $R$ can be graded by defining the \textit{multidegree} of a monomial $ \mathsf{x}=\prod_{\rho \in
\Sigma(1)}x_{\rho}^{a_{\rho}} $ to be $ \left[ \mathbb{V}\left( \mathsf{x} \right) \right] \in
A^1(T_{\Sigma})$, where $A^1(T_{\Sigma})$ denotes the codimension-one Chow group of $T_\Sigma$. Setting $ R_{\alpha}=\bigoplus_{\mathsf{x}:\;\left[ \mathbb{V}\left( \mathsf{x} \right) \right] =\alpha} k \cdot \mathsf{x}$ we have
that $R=\bigoplus_{\alpha \in A^1(T_{\Sigma})} R_{\alpha}.$ We say that a
polynomial $f\in R$ is \textit{homogeneous} if it is homogeneous with respect to
this grading, i.e., if all monomials in $f$ have the same multidegree. An ideal in $R$ is
called homogeneous if it is generated by homogeneous polynomials. When $T_\Sigma=\pp^n$ the Cox ring is simply the standard graded coordinate ring of $\pp^n$ and the multidegree of a monomial is simply the total degree of the monomial multiplied by the class of a general hyperplane. More details can be found in the book \cite{cox2011toric}. 
\subsubsection{Homogeneous generators}\label{subsubsec:HomGens}
Let $X \subset T_\Sigma$ be a closed subscheme.
We can always find an ideal $I = \langle f_0, \dots, f_r \rangle \subset R$ defining $X$ so 
that $[\mathbb{V}(f_i)]=\alpha$ for all $i$, for some fixed $\alpha$ in
$A^1(T_\Sigma)$ (see \cite[6.A]{cox2011toric}).
Using this set of generators we see that $X$ is the base scheme of the linear system defined by
$f_0,\dots, f_r$, viewed as sections of $\mathcal{O}(\alpha)$. 
\begin{definition}\label{def:alpha-homogeneous}
We say a homogeneous polynomial $f_i$ in the Cox ring of a smooth projective toric variety $T_\Sigma$ has \emph{multidegree} $\alpha=[\mathbb{V}(f_i)]\in A^1(T_{\Sigma})$ and say a set of polynomials $f_0,\dots,f_r$ all having the same multidegree $\alpha$ is \emph{$\alpha$-homogeneous}.
\end{definition}
\begin{convention}
Let $X$ be a subscheme of $T_\Sigma$ defined by an ideal $I$ generated by polynomials $f_0,\dots, f_r$, we assume (without loss of generality) that this set of polynomials is $\alpha$-homogeneous. We will use this convention for the defining equations of \textit{all} subschemes/subvarieties considered in this paper unless otherwise stated.
\end{convention}

\begin{remark}
In the case where $T_\Sigma=\PP^n$ we are simply assuming that a given set of polynomial generators have the same degree. If we are given an ideal in the coordinate ring of $\pp^{n_1}\times\cdots \times \pp^{n_m}$ where the generators do not have the same multidegree we may construct a new ideal which has $\alpha$-homogeneous generators and defines the same scheme as follows. 

Work in $\pp^{n_1}\times\cdots \times \pp^{n_m}$ with multi-graded coordinate ring 
$R=k[x^{(1)},\dots,x^{(m)}]$, where $x^{(j)}=x_1^{(j)},\dots,x_{n_j}^{(j)}$. In this case the Chow group
$A^1(\pp^{n_1}\times\cdots \times \pp^{n_m})$ is generated by $h_1,\dots, h_m$ where $h_i$ is the 
pullback of the hyperplane class in the factor $\pp^{n_i}$. The multidegree of a monomial
$\mathbf{x}^{\mathbf{a}}$ in $R$ has the form $d_1h_1+\cdots +d_mh_m$ where $d_i$ is the total degree of
$\mathbf{x}^{\mathbf{a}}$ in the variables $x^{(i)}$. Let $B_j=\left\langle x_1^{(j)},\dots,x_{n_j}^{(j)}\right\rangle$, so that the irrelevant ideal of $R$ has
primary decomposition $B=B_1\cap \cdots \cap B_m$,
and for $d\in \mathbb{N}$, let $B_j(d)$ be the ideal generated by the $d$-th
powers of the generators of $B_j$.

Take a subscheme $X\subset T_\Sigma$ defined by a homogeneous ideal $I= \langle w_1,\dots, w_l \rangle$ in $R$ with the
generator $w_i$ having multidegree $d^{(i)}_1h_1+\cdots+d^{(i)}_mh_m$.
Let $D_i = \max_{1 \leq j \leq l} (d^{(j)}_i)$ for $1 \leq i \leq m$.
We can construct a new ideal $J$ for $X$ with generators all having multidegree $D = D_1h_1 + \cdots+ D_m h_m$; the ideal $J$ is given by:
  \[
 J= \sum_{i=1}^m \langle w_i \rangle \cdot B_1(D_1-d^{(i)}_1)\cdots B_m(D_m-d^{(i)}_m).
\]Note that in the equation above we use summation notation for the sum of {ideals} in $R$.
The reader can verify that $I:B^\infty = J:B^\infty$, meaning that $I$ and $J$ define the same subscheme. \label{remark:SameMultiDegreeGens}
\end{remark}

\begin{example}
The procedure discussed in Remark \ref{remark:SameMultiDegreeGens} above is easy to apply by hand.  
Work in $\pp_x^2\times \pp_y^3$ with coordinate ring $R=k[x_0,\dots,x_2,y_0,\dots,y_3]$ and codimension-one Chow
group $A^1(\pp_x^2\times \pp_y^3)$ generated by the hyperplane classes $h_x,h_y$.
The irrelevant ideal of the coordinate ring $R$ is $B= \langle x_0,x_1,x_2 \rangle \cdot \langle y_0,y_1,y_2,y_3 \rangle.$

Consider the ideal
$I=\langle x_0x_1^2y_1-x_2^3y_3,x_2y_2^2-x_1y_0y_1 \rangle$ defining a scheme $X$.
The generators of $I$ have multidegrees $3h_x+h_y$ and $h_x+2h_y$, respectively.
The new ideal $J$ will have generators all of multidegree $3h_x + 2h_y$:
\[ J = \langle  x_0x_1^2y_1-x_2^3y_3 \rangle \cdot \langle y_0,y_1,y_2,y_3 \rangle +
                   \langle x_2y_2^2-x_1y_0y_1 \rangle \cdot \langle x_0^2, x_1^2, x_2^2 \rangle\\
\]
and the ideal $J$ also defines the scheme $X\subset \pp_x^2\times \pp_y^3$ since $I:B^\infty=J:B^\infty$.
\end{example}

To construct an $\alpha$-homogeneous set of generators from a given set of homogeneous generators for a subscheme of an arbitrary smooth projective toric variety, a technique similar to that of Remark \ref{remark:SameMultiDegreeGens} can be used. Instead of multiplying by powers of the generators of components of the irrelevant ideal, we multiply the given generators by powers of the generators of products of components of the irrelevant ideal of the Cox ring. Since the fan of a general smooth projective toric variety is more combinatorially complicated than that of a product of projective spaces, the procedure is also more difficult to write down, but is otherwise similar.

\subsubsection{Multi-indices}
We will make frequent use of standard multi-index notations throughout the paper.  In particular
for a non-negative integer vector $a=(a_1,\dots,a_m)$, we have $|a|=a_1{+}\cdots{+}a_m$, and if $x_1,\dots,x_m$
are the variables of a ring, we write $x^a = x_1^{a_1} \cdots x_m^{a_m}$.  Notice that if $h_1,\dots,h_m$ are
generators of the Chow ring $A^*(\pp^{n_1}\times\cdots\times\pp^{n_m})$, then $h^a$ is a class of
codimension $|a|$.
\subsection{Segre classes}
In this subsection we define the Segre class of a subscheme and summarize computational methods for Segre classes $s(X,Y)$ in the Chow ring $A^*(\pp^n)$.

For $X$ a subscheme of a scheme $Y$ the Segre class $s(X,Y)=s(C_XY)$ is the Segre class of the normal cone $C_XY$ to $X$ in $Y$ (see \cite[\S4.2]{fulton2013intersection} for more details). In the case where $Y$ is a variety we may define the Segre class via Corollary 4.2.2 of \cite{fulton2013intersection}. 
\begin{definition}
  Let $X$ be a closed subscheme of a variety $Y$.  We have a blowup diagram
  \[\begin{tikzcd}
    E \ar[d,"\eta"] \ar[r] \arrow[dr, phantom, "\square"] & \mathrm{Bl}_X Y \ar[d,"\pi"] \\
    X \ar[r] & Y \rlap{\ ,}
    \end{tikzcd} \]
    where $E$ is the exceptional divisor. The \emph{Segre class} of $X$ in $Y$ is
  \[
    s(X,Y) = \eta_*((1 - E + E^2 - \dots) \frown [E]) \in A_*(X),
  \]
  see \cite[Corollary~4.2.2]{fulton2013intersection}.  
  When $ Y$ is contained in some smooth projective toric variety $T_\Sigma$ we will frequently abuse
  notation and write $s(X,Y)$ for the pushforward to $A^*(T_\Sigma)$. \label{def:Segre}
\end{definition}
We now review the computation of the class $s(X,Y) \in A^*(\pp^n)$ in the special case where $X$ is a
subscheme of a projective variety $Y \subset \pp^n$. In this case we may think of $X$ as the base scheme
of an $(r+1)$-dimensional linear system of global sections of $\mathcal{O}_Y(d)$. In practice this means
choosing a set of $r+1$ (scheme-theoretic) generators for $ X$ all of the same degree $d$, say $f_0,\dots,f_r
\in k[x_0,\dots,x_n]$.
Consider the graph
\begin{equation}
\begin{tikzcd}
  & \mathrm{Bl}_X Y \ar[dr,"\rho"] \ar[dl,swap,"\pi"] \arrow[draw=none]{r}[sloped,auto=false]{\subset} & \pp^n \times \pp^r \\
  Y \ar[rr,dashed,"\pr_X"] & & \pp^r 
\end{tikzcd}\label{eq:SegreDiagBackGrd}
\end{equation} 
of the rational map $\pr_X : Y \dashrightarrow \pp^r$ defined by $\pr_X :p\mapsto (f_0(p):\cdots: f_r(p))$.
\begin{definition} We call $pr_X$ the \emph{projection of $Y$ along $X$}.
The \emph{projective degrees} of $\pr_X$ are the non-negative integers
\[
  g_i(X,Y)= \int h^i \cdot \left[\overline{\pr_X^{-1}(\pp^{r-(\dim(Y)-i)}) - X} \right],
\]
where $h$ is the hyperplane class in the Chow ring $A^*(\pp^n)\cong \ZZ[h]/\langle h^{n+1} \rangle$. 
\label{def:projDegreesPPn}
\end{definition}

\begin{proposition}[{\cite[Prop. 5]{Harris2017}}]
  \label{prop:segreFormulaPPn}
  The Segre class $s(X,Y) \in A^*(\pp^n)$ is given by
  \[ s_i = \sum_{j=0}^{\dim X} \binom{\dim Y - i}{j - i} (-d)^{j-i}(d^{\dim Y}\deg(Y) - g_j(X,Y)). \]
\end{proposition}

In \cite[Theorem 4.1]{helmer2016proj}, the projective degrees of a rational map $\Phi: \pp^n \dashrightarrow \pp^r$
are expressed as the dimensions of a sequence of finite-dimensional $k$-algebras.  There is an analogous result
for our more general situation.  The projective degrees $g_i(X,Y)$ for $\pr_X: Y \dashrightarrow \pp^r$ can be computed
directly from the definition above
as the degree of the 0-dimensional variety
\[
 Z_i=(Y \cap L^{i} \cap\mathbb{V}(P_1,\dots,P_{\dim(Y)-i})) - X
\]
where $L^{i}$ is a general linear space of codimension $ i$
in $\pp^n$ and $P_j$ is a general $k$-linear combination of $\alpha$-homogeneous generators of $\mathcal{I}_X$. Note $\mathbb{V}(P_1,\dots,P_{\dim(Y)-i}))=\pr_X^{-1}(\pp^{r-(\dim(Y)-i)})$ by construction.

We now move to $\pp^n \times \bbA^1$ and restrict to the affine open
subset $D(L')$ defined by the non-vanishing of a general linear form $L'$.
Further, we want to remove $X$, so we restrict to $D(L') \cap D(F')$ 
where $F'$ is a general $k$-linear combination of $\alpha$-homogeneous generators for $\mathcal{I}_X$.
The resulting affine variety $$((Y \cap L^{i} \cap \mathbb{V}(P_1,\dots,P_{\dim Y -i})) \times \bbA^1) \cap D(F') \cap D(L')$$ is in
1-1 correspondence with the set of points $Z_i$.  This leads to
the following expression for the projective degrees.

\begin{proposition}[cf.~{\cite[Theorem 4.1]{helmer2016proj}}]
  \label{prop:projDegreesPPn}
  The projective degrees $g_i$ of $\pr_X$ are given by
  \[ g_i(X,Y) = \dim_k \frac{k[x_0,\dots,x_n,T]}{\mathcal{I}_Y + \mathcal{I}_{L^{i}} +  \langle P_1,\dots, P_{\dim(Y)-i}, 1 - T \cdot F', 1 - L' \rangle }, \]
 where $L'$ is a general affine linear form in the $x_i$ and $F'$ is a general $k$-linear combination of $\alpha$-homogeneous generators for $\mathcal{I}_X$.
\end{proposition}

\subsection{Projective degrees of $X$ in $Y$}\label{subsec:ProjectiveDegrees}
Let $T_{\Sigma}$ be a smooth projective toric variety defined by a fan $\Sigma$ 
with Cox ring $R$ (see \S\ref{subsubsec:CoxRing}). Let $ Y \subset T_\Sigma$ be a pure-dimensional subscheme and 
let $X \subset Y$ be a subscheme defined by an $\alpha$-homogeneous ideal $\mathcal{I}_X=\langle f_0,\dots,f_r \rangle \subset R$ (see Definition \ref{def:alpha-homogeneous}).

\begin{definition}
  We define the \emph{projection of $Y$ along $X$} to be the rational map $\pr_X:Y \dashrightarrow \pp^r$ given by
\begin{equation}
\pr_X: p \mapsto (f_0(p):\cdots: f_r(p)).\label{eq:rational_map_of_ideal}
\end{equation} \end{definition}Let $\Gamma$ denote the blowup of $Y$ along $X$. We have the following diagram: 
\begin{equation}
\begin{tikzcd}
  & \Gamma \ar[dr,"\rho"] \ar[dl,swap,"\pi"] \arrow[draw=none]{r}[sloped,auto=false]{\subset} & T_\Sigma \times \pp^r \\
  Y \ar[rr,dashed,"\pr_X"] & & \pp^r \rlap{\ .}
\end{tikzcd}\label{eq:SegreGraphDiagram}
\end{equation}
Now we define the class
\begin{equation}
G(X,Y) := \sum_{{i}=0}^{\dim(Y)} \pi_*(K^{i} \cdot [\Gamma])  \in A^*(T_\Sigma), 
\end{equation}
where $K$ is the pullback along $\rho$ of the hyperplane class.
This class is a minor generalization of Aluffi's \emph{shadow of the graph} (\cite{Aluffi2003c}).
By construction
\[ \pi_*(K^{i}\cdot [\Gamma])=\left[ \overline{\pr_X^{-1}(\pp^{r-{(\dim(Y)-i)}}) - X} \right],\]
where $\pp^{r-{(\dim(Y)-i)}} \subset \pp^r$ is a general linear subspace.

We may write $G(X,Y)$ more explicitly as follows. Let $b_1,\dots, b_{m} \in A^1(T_{\Sigma})$ be a fixed nef basis for $A^1(T_{\Sigma})$ (this exists since $T_\Sigma$ is projective, see \cite[Proposition 6.3.24]{cox2011toric}). 
We may express the rational equivalence class of a point in $T_\Sigma$ as $b^n=b_1^{n_1}\cdots b_{m}^{n_m}$ where $n_j > 0$ for all $j$, and
the degree-$i$ monomials in $b_1,\dots, b_{m}$ which divide $b^n$ in $A^*(T_\Sigma)$ form a monomial basis for $A^i(T_\Sigma)$.
Hence we may write 
\begin{equation}
G (X,Y)=\sum_{|\nu|\leq \dim(Y)} g_{\nu}(X,Y) \cdot b^{n-\nu} \in A^*(T_{\Sigma}), \label{eq:G_projDeg}
\end{equation} where $\nu=(\nu_1,\dots, \nu_m)$.  
Note that the indices $\nu$ appearing in this expression depend on the choice of representative $b^n$ of the point class, which is not unique in general; however the class $G(X,Y)$ {\em does not} depend on this choice.
\begin{definition}
  We refer to the coefficients $g_{\nu}(X,Y)$ as the \textit{projective degrees of $X$ in $Y$}.
  In other words the projective degrees are
\begin{equation}
g_{\nu}(X,Y)=\int b^{\nu} \cdot \left[ \overline{\pr_X^{-1}\left(\pp^{r-(\dim(Y)-{|\nu|})}\right)
 - X} \right].\label{eq:ProjDegsExplicit_No_G}
\end{equation}
\label{def:projDegrees}
\end{definition}
The class $G(X,Y)$ measures, in a sense, how algebraically dependent the polynomials
$f_0,\dots,f_r$ are in $Y$.  If the linear system were base-point free on $Y$ (so $X \not\subset Y$ and $Y \not\subset X$), the shadow
of the graph would be $\sum_{i=0}^{\dim(Y)} \alpha^{\dim(Y)-i} [Y]$.  
The difference of these two classes will play an important role in the computation of the Segre class (see Theorems \ref{MainTheorem2MultiProj} and \ref{thm:MainSegreToric}), so we define the following notation:
\begin{equation}
  \Lambda (X,Y) =\sum_{|a|\leq \dim(X)}\Lambda_{a}(X,Y) \cdot b^{n-a}
:= \sum_{i=0}^{\dim(Y)} \alpha^{\dim(Y)-i}[Y] - G(X,Y) \in A^*(T_{\Sigma}).
  \label{eq:Lambda}
\end{equation}
In particular, $\Lambda_a(X,Y)=0$ when $|a| > \dim(X)$.

\begin{example}[Projective degrees in $\pp^n$] Work in $\pp^2_x$ with Chow ring $A^*(\pp^2_x)=\ZZ[h]/\langle h^3\rangle$ and consider the varieties $X=\mathbb{V}(x_0,x_1)$ and $Y=\mathbb{V}(x_0^3 + x_0^2 x_2 - x_1^2 x_2)$.
Using Theorem \ref{MainTheorem1_Multi_proj} we compute that the projective degrees are $g_0=7$ and $g_1=3=\deg(Y)$. Hence we have that $G(X,Y)=7h^2+3h$. For this example $[Y]=3h$ and we let $\alpha=3h$. Substituting these values into \eqref{eq:Lambda} we obtain $$\Lambda(X,Y)=(9h^2+3h)-(7h^2+3h)=2h^2.$$
\end{example}
\begin{example}[Projective degrees in products of projective spaces]Work in $\pp_x^2\times \pp_y^3$ with Chow ring $A^*(\pp_x^2\times \pp_y^3)\cong \ZZ[h_1,h_2]/\langle h_1^3,h_2^4 \rangle$. In this Chow ring the class of a point is $[pt]=h_1^2h_2^3$.
Consider the 3-dimensional variety  $$Y=\mathbb{V}(x_0x_2y_0-x_1^2y_2,y_3)$$
with divisor $X \subset Y$ defined by $x_1y_2+x_0y_0=0$.
For this example,  $[Y]=2h_1h_2+h_2^2$, and 
and we let $\alpha=2h_1+h_2$.
Computing the projective degrees using Theorem \ref{MainTheorem1_Multi_proj} we obtain the following. In dimension zero we have $g_{(0,0)}=0$. In dimension one we have $g_{(1,0)}=0$ and $g_{(0,1)}=1$. In dimension two we have $g_{(2,0)}=0$, $g_{(1,1)}=1$, and $g_{(0,2)}=2$. In dimension three we have $g_{(2,1)}=1$, $g_{(1,2)}=2$, and $g_{(0,3)}=0$. Hence we have that\footnotesize \begin{align*}
 G(X,Y)&=(0\cdot h_1^2h_2^3)+(0\cdot h_1h_2^3+1\cdot h_1^2h_2^2)+(0\cdot h_2^3+1\cdot h_1h_2^2+2\cdot h_1^2h_2)+(1\cdot h_2^2 + 2\cdot h_1h_2+0\cdot h_1^2)\\
 &=h_1^2h_2^2+h_1h_2^2+2h_1^2h_2+h_2^2+2h_1h_2.   
\end{align*}\normalsize Substituting these values into \eqref{eq:Lambda} we obtain \begin{align*}
    \Lambda(X,Y)&=(24h_1^2h_2^3+12h_1^2h_2^2+6h_1h_2^3+4h_1^2h_2+4h_1h_2^2+h_2^3+2h_1h_2+h_2^2)-G(X,Y)\\
    &=24h_1^2h_2^3+11h_1^2h_2^2+6h_1h_2^3+2h_1^2h_2+3h_1h_2^2+h_2^3.
\end{align*}

\end{example}

\begin{example}[Projective degrees in a toric variety]
Work in the smooth Fano toric threefold\footnote{The variety $T_\Sigma$ is generated in Macaulay2 \cite{M2} with the command \texttt{smoothFanoToricVariety(3,2)} from the {\tt NormalToricVarieties} package.} 
 $T_\Sigma$ where $\Sigma$ is the fan with rays $\rho_0{=}(1, 0, 0)$, $\rho_1{=}(0, 1, 0)$, $\rho_2{=}(-1, -1, -1)$, $\rho_3{=}(0, 0, 1)$, $\rho_4{=}(0, 0, -1)$ and maximal cones $\left\langle\rho_0, \rho_1, \rho_3\right\rangle$, $\left\langle\rho_0, \rho_1, \rho_4\right\rangle$, $\left\langle\rho_0, \rho_2, \rho_3\right\rangle$, $\left\langle\rho_0, \rho_2, \rho_4\right\rangle$, $\left\langle\rho_1, \rho_2, \rho_3\right\rangle$, $\left\langle\rho_1, \rho_2, \rho_4\right\rangle$.
Let $\overline{O(\sigma)}$ denote the orbit closure of a cone $\sigma\in \Sigma$ and set $b_1=[\overline{O(\rho_3)}]\in A^1(T_\Sigma)$ and $b_2=[\overline{O(\rho_0)}]=[\overline{O(\rho_1)}]=[\overline{O(\rho_2)}]\in A^1(T_\Sigma)$. The divisors $b_1,b_2$ form a nef basis for $A^1(T_\Sigma)$. The Cox ring of $T_\Sigma$ is $R=\CC[x_0,\dots, x_4]$, with irrelevant ideal $B=(x_0,x_1,x_2)\cap (x_3,x_4)$. The Cox ring is $\ZZ^2$ graded; the multidegree of $x_3$ is $b_1$, the multidegree of $x_0$, $x_1$ and $x_2$ is $b_2$, and the multidegree of $x_4$ is $b_1-b_2$. The Chow ring can be written as
\[
A^*(T_\Sigma)\cong \ZZ[b_1,b_2]/\left\langle b_1^3,2b_2-2b_1,b_1b_2-b_2^2 \right\rangle.
\]
For more on constructing the Chow ring of a smooth projecive toric variety from its fan see \cite[Theorem~10.8]{Danilov1978}.  

Let $X=\mathbb{V}(5x_0+7x_2,x_3+x_2x_4)\subset T_\Sigma$ and let $Y=\mathbb{V}(x_3+x_2x_4)\subset T_\Sigma$. Then $X$ is a curve on the surface $Y$. For this example, $[Y]=b_1$ and we let $\alpha=b_1+b_2$.
Let $[pt]$ denote the rational equivalence class of a point in $A^*(T_\Sigma)$. Using Proposition \ref{propn:projective_degreeToric} we compute the projective degrees and obtain 
$
G(X,Y)=[pt]+b_1b_2+b_1.
$
Substituting $\alpha=b_1+b_2$ and $[Y]=b_1$ into \eqref{eq:Lambda} we obtain $$
\Lambda(X,Y)=([pt]+b_1^2+b_2^2+b_1)-([pt]+b_1b_2+b_1)=b_1^2\in A^*(T_\Sigma)
.$$
Note that $[pt]=b_2^3=b_1b_2^2 = b_1^2b_2$.  The lack of a unique representative for the point class in this case stems from the existence of non-effective divisors on $T_\Sigma$. Since the indexing convention for projective degrees depends on the chosen representative of the point class, we have opted to simply write the resulting unique class $G(X,Y)$.
\end{example}
\begin{remark}
  If $L^a$ is the complete intersection of general divisors of $T_{\Sigma}$ with $[L^a]=b^a$ such that $|a| \leq \min\{\dim(X),|\nu|\}$,
  then we have that \[g_\nu(X,Y)=g_{\nu-a}(X\cap L^a,Y\cap L^a). \]
  In particular, if $D$ is a divisor with $[D] =b_i \in A^1(T_\Sigma)$, then $g_1(X,Y) = g_0(X \cap D, Y \cap D)$. 
  \label{remark:projDegLinSlice}
\end{remark}
 
\section{Computing Segre classes}
\label{sec:MainResults}
In this section we give the main results of this paper, namely formulas for the Segre class in terms of projective degrees (\S\ref{subsec:ProjectiveDegrees}) and explicit formulas which allow us to compute these projective degrees effectively in practice. We first give our results for subschemes of a product of projective spaces. In \S\ref{subsec:SegSubToric} we generalize this to subschemes of a smooth projective toric variety. Throughout this section we will freely use the notations and conventions defined in \S\ref{subsec:notationConvention}.

\subsection{Computing Segre classes in $A^*(\pp^{n_1} \times\cdots \times \pp^{n_m})$}\label{subSec:SegreMultiProj}We begin by defining some notations for products of projective spaces. 
\begin{definition}\label{def:chowring}
Throughout this section we work in $\pp^{\mathbf{n}}_m := \pp^{n_1}\times\cdots \times \pp^{n_m}$ with multi-graded coordinate ring 
$R=k[x^{(1)},\dots,x^{(m)}]$, where $x^{(j)}=x_1^{(j)},\dots,x_{n_j}^{(j)}$.
The Chow ring of $\pp^{n_1}\times\cdots \times \pp^{n_m}$ is
\begin{equation}
A^*(\pp^{\bf n}_m)\cong \ZZ[h_1,\dots,h_m]/\langle h_1^{n_1+1},\dots,h_m^{n_m+1} \rangle\label{eq:ChowRingMultiProj}
\end{equation}
so that $h_j$ is (the pullback of) the class of a hyperplane defined by a general linear form in the variables $x^{(j)}$.
It is sometimes convenient to write a class by dimension, instead of by codimension.  To this end, for
$a=(a_1,\dots,a_m)$ we write \[ [\pp^a_m] := h^{n-a} = h_1^{n_1-a_1} \cdots h_m^{n_m-a_m} , \] 
and we write $L^a$ for a general linear space in $\pp^\mathbf{n}_m$ such that
\[ [L^a] = h^a . \]

\end{definition}
\begin{convention}
In this section, we consider $X\subset Y \subset \pp^{\bf n}_m$ with the following conventions:
\begin{enumerate}
\item $Y \subset \pp^{\bf n}_m$ is an irreducible scheme of dimension $N = \dim(Y)$,
\item $X \subset Y$ is a closed subscheme defined by $\alpha$-homogeneous polynomials $f_0,\dots,f_r$.
\end{enumerate}\label{conv:X_Y_MultiProj}
\end{convention}
Recall from \S\ref{subsubsec:HomGens} that the assumption (2) above is made without loss of generality.
\begin{proposition}
The projective degrees of $X$ in $Y$ are given by
\[ g_{a}(X,Y) = \deg \left( Y \cap L^a \cap W - X \right),\]
 where \begin{equation}
W=\mathbb{V}(P_1,\dots,P_{\dim(Y)-|a|}), \;\;\;\mathrm{with}\;\;\; P_j=\sum_i \lambda_i f_i\label{eq:Pjdef}
\end{equation} 
for general $\lambda_i\in k$. 
\label{propn:projective_degreeFirstMultiProj}
\end{proposition}
\begin{proof}
By definition, $g_a(X,Y) = \int b^{n-a} \cdot [\overline{pr_X^{-1}(\mathfrak{L})-X}]$, where $\mathfrak{L} = \pp^{r-(\dim(Y)-|a|)} \subset \pp^r$ is general in the sense of Kleiman's transversality theorem \cite{kleiman1974transversality}.  Then $W=pr_X^{-1}(\mathfrak{L})$ is the complete intersection of $\dim(Y)-|a|$ hypersurfaces defined by general $k$-linear combinations of the $f_i$ and $W-X$ is reduced.
\end{proof}

\begin{remark} 
Let $C_XY$ denote the normal cone to $X$ in $Y$.
The components $C_1,\dots,C_l$ of the normal cone project onto subvarieties $Z_1,\dots,Z_l$ of $X$. These are known as the \emph{distinguished subvarieties}; see also \cite[Definition~6.1.2]{fulton2013intersection}. 

Consider the blowup diagram
   \[\begin{tikzcd}[baseline=(b.base)]
       & \Gamma \ar[dr,"\rho"] \ar[dl,swap,"\pi"] & \\
       Y \ar[rr,dashed,"\pr_X"] & & |[alias=b]| \pp^r \rlap{.}
     \end{tikzcd} \]
Then $\Gamma \subset \pp^{\bf n}_m \times \pp^r$ has class 
$$[\Gamma] = \sum_a g_a(X,Y) [\pp^a_m \times \pp^{\dim(Y)-|a|}].$$
Let $L^a \subset \pp^{\bf n}_m$ and let $W(X,Y)= \overline{\pr_X^{-1}\left(\pp^{r-(\dim(Y)-|a|)}\right)}$. 
In order to compute the projective degrees via Definition \ref{def:projDegrees}, we take $\deg(\pi^{-1}(L^a) \cap \rho^{-1}(\pp^{r-(\dim(Y)-|a|)}) \cap \Gamma)$, and this will equal $g_a(X,Y)$, provided that our choices are \emph{general enough}.
The precise condition is that $L^a$ must not contain the distinguished subvarieties for $X$ in $Y$, or else $\pi^{-1}(L^a)$ will contain components of the exceptional divisor, and therefore fail to be of the correct codimension in $\Gamma$.  
Since we perform the computation in $\pp^{\bf n}_m$, the condition on $W$ is simpler; it must meet $L^a \cap Y$ in dimension 0.  
This is a simpler point of view than that of \cite[Theorem 3.2]{Harris2017}. In implementations, ``general'' is replaced by some version of ``random''; however, using this remark we can verify that our choices are general enough, which makes the resulting algorithms non-probabilistic.\label{remark:GeneralChoiceOFLinSpace}
\end{remark}

We now give an algebraic version of the geometric result above.  This formulation is what allows for numerical approaches to be applied to the computations in this paper.

\begin{theorem}
The projective degrees of $X$ in $Y$ are given by
  \[ g_a(X,Y)= \dim_k \frac{R[T]}{\mathcal{I}_Y + \mathcal{I}_{L^a} + A + \langle P_1,\dots,P_{N-|a|},1-T\cdot P_0 \rangle }, \]
where $A=\langle \ell(x^{(1)})-1,\dots,\ell(x^{(m)})-1 \rangle$ for $\ell(x^{(i)})$ a general linear form in $x^{(i)}$, and $P_j$ for $0 \leq j$ is as in \eqref{eq:Pjdef}.
\label{MainTheorem1_Multi_proj}
\end{theorem}\begin{proof}
From Proposition \ref{propn:projective_degreeFirstMultiProj} we know that
\[
g_a(X,Y)=\deg(Y\cap L^a \cap W-X)
\]
where $V_1=(Y\cap L^a \cap W)-X$ is a zero-dimensional subscheme of $\PP^{n_1}\times \cdots\times \pp^{n_m}$. The general (homogenous) linear forms in $x^{(i)}$ have the form $\ell(x^{(i)})=\theta_0x_0^{(i)}+\cdots +\theta_{n_i}x_{n_i}^{(i)}$ for general $\theta_j\in k$. We can dehomogenize each factor $\pp^{n_i}$ by choosing a general hyperplane at infinity defined by the affine linear form $\ell(x^{(i)})-1$ (instead of, for example, the hyperplane defined by $x_0^{(i)}-1$). The ideal $A$ then dehomogenizes $\PP^{n_1}\times \cdots\times \pp^{n_m}$ and gives an affine scheme 
\[ 
\tilde{V}_1=(Y\cap L^a \cap W-X)\cap \mathbb{V}(A)\subset k^{n_1+1}\times \cdots \times k^{n_m+1}
\]
such that $\deg(V_1)=\deg(\tilde{V}_1)$ and that $(L^a \cap W-X)\cap \mathbb{V}(A)$ is reduced. Linearly embed \[k^{n_1+1} \times \dots \times k^{n_m+1} \hookrightarrow (k^{n_1+1} \times \dots \times k^{n_m+1}) \times k^1=\mathrm{Spec}(R[T])\] 
and consider its image
\[
\tilde{V}_2=Y\cap L^a \cap W\cap \mathbb{V}(A)\cap \mathbb{V}(1-TP_0)\subset \mathrm{Spec}(R[T]),
\] written using the Rabinowitsch trick, by which it follows that $\deg(\tilde{V}_2)=\deg(\tilde{V}_1)$.
\end{proof}

Let $i = (i_1,\dots,i_m)$, and write the Segre class as
\begin{eqnarray}
s(X,Y) \;\;=\sum_{|i|\leq \dim(X)}
    s_{i}(X,Y) \cdot h^{n-i}.\label{eq:SegreCoeffs}
\end{eqnarray} 
We now give a recursive formula for the Segre class in terms of projective
degrees.  
Note that the recursion begins from $|a|=\dim(X)$ and proceeds to $|a|=0$. 
\begin{theorem}
The Segre class of $X$ in $Y$ is given by the recursive formula
 \begin{equation}\label{eq:segreRecursion}
s_a(X,Y)=\Lambda_{a}(X,Y) -\int (1+\alpha)^{N-|a|} h^a \cdot
\sum_{|i| > |a|} s_i(X,Y) h^{n-i},
\end{equation}
where $\Lambda_a(X,Y)$ is the coefficient defined in \eqref{eq:Lambda}.
\label{MainTheorem2MultiProj}\label{propn:SegreSchemeY}
\end{theorem}

\begin{proof}
  First suppose that $Y$ is a variety. 
	To write $s(X,Y)$ in terms of $G(X,Y)$, we use
  Proposition 4.4 of \cite{fulton2013intersection} to obtain the
  relation
  \[ g_{0} = \int c_1(\mathcal{O}(\alpha))^{N}\frown [Y] -\int c(\mathcal{O}(\alpha))^{N} \frown s(X,Y). \]
  Let $a=(a_1,\dots,a_m)$. 
  By Remark
  \ref{remark:projDegLinSlice} we get
  \[ g_{a}(X,Y)= g_{0}(X\cap L^a,Y \cap L^a)= \int \alpha^{N-|a|} {h}^{a} [Y] -\int (1 + \alpha)^{N-|a|}
    {h}^{a} \frown s(X,Y) ,\] where $L^a$ is a general linear space.
  Using this expression for the projective degrees and summing over $a$, we obtain that
  \begin{equation}
    G(X,Y)= \sum_{i=0}^{\dim(Y)}\alpha^{N-i}[Y] - \sum_{|a|\leq \dim(Y)} \left( \int (1 + \alpha)^{N-|a|}
    {h}^{a} \frown s(X,Y) \right) h^{n-a}, \label{eq:segProofFormula}
  \end{equation}
  and when $|a| > \dim(X)$ the last term in \eqref{eq:segProofFormula} is zero, so we can write
  \begin{equation}
    \Lambda(X,Y)= \sum_{|a|\leq \dim(X)} \left( \int (1 + \alpha)^{N-|a|}
    {h}^{a} \frown s(X,Y) \right) h^{n-a}. 
  \end{equation}
In the first step, when $|a|=\dim(X)$, we get $s_{a}(X,Y) = \Lambda_{a}(X,Y)$ and proceeding with $|a|=\dim(X)-1,\dots, 0$ gives \eqref{eq:segreRecursion}.

Now we return to the general hypothesis of $Y$ an irreducible scheme.
Let $\mathfrak{m}$ denote the geometric multiplicity of $Y$ so that $[Y]=\mathfrak{m} [Y_{\mathrm{red}}]$.
By Fulton \cite[Lemma~4.2]{fulton2013intersection} we have that
\[
  s(X,Y)=\mathfrak{m}\cdot s(X \cap Y_\mathrm{red},Y_{\mathrm{red}}).
\]
Projective degrees count points (with geometric multiplicity) in a zero-dimensional set, and we have
\begin{align*}
  g_a(X,Y) &= \deg\left( Y \cap L^a \cap W - X \right) \\
  &= \mathfrak{m} \cdot \deg\left( Y_{\rm red} \cap L^a \cap W - X \right) \\
  &= \mathfrak{m} \cdot g_a(X \cap Y_{\rm red}, Y_{\rm red}).
\end{align*}
Substituting this into $\Lambda(X,Y)$ gives
  \begin{align*}
  \Lambda(X,Y)  &= \sum_{i=0}^{\dim(Y)}\alpha^{N-i} [Y] - G(X,Y) \\
  &= \sum_{i=0}^{\dim(Y)}\alpha^{N-i} (\mathfrak{m}[Y_{\rm red}]) - \mathfrak{m}\cdot G(X \cap Y_{\rm red},Y_{\rm red}) \\
  &= \mathfrak{m} \cdot \Lambda(X \cap Y_{\rm red}, Y_{\rm red})\;.
  \end{align*}
The proof is completed by observing that $\mathfrak{m}$ also factors out of the right-hand side of \eqref{eq:segreRecursion}.
\end{proof}
  
We now prove a lemma which gives the value of the right-hand side of the
expression for the dimension-$X$ projective degrees
when $X$ is not contained in $Y$. Using this,
we show in Proposition \ref{corr:LambdaDimXReducible} that the dimension-$X$ part of the Segre class $s(X,Y)$
can be computed directly (i.e.,~without knowing the irreducible components of $Y$) when $Y$ is a 
reducible scheme and has pure dimension. This fact will be useful in \S\ref{subSection:AlgMult1} and \S\ref{sec:GBFreeContainment}.  
\begin{lemma}
  Let $X\subset  \pp^{\mathbf{n}}_m$ be a closed subscheme defined by an $\alpha$-homogeneous ideal.
  Let $Z \subseteq \pp^{\mathbf{n}}_m$ be an irreducible scheme such that  $Z_{\rm red} \not\subset X$ and $(X_i)_{\rm red} \not\subset Z$ for all top-dimensional irreducible components $X_i$ of $X$. 
For $a\in \ZZ_{\geq 0}^m$ such that $|a|=\dim(X)$, 
we have that
\[ 
g_a(X \cap Z, Z) =\deg \left( Z \cap L^a \cap W  \right)
= \int h^a \cdot \alpha^{N - \dim(X)} [Z].
\] 
where $W$ is as in \eqref{eq:Pjdef};
or equivalently,
$ \Lambda_a(X \cap Z, Z) = 0.$
\label{propn:reducible_ProjDeg}
\end{lemma}
\begin{proof}
By Proposition \ref{propn:projective_degreeFirstMultiProj} we have
that \[g_a(X \cap Z, Z)=\deg \left(Z \cap L^a \cap W -(X\cap Z) \right),\]
Since $Z$ is irreducible, the non-containment assumptions imply that $\dim(X \cap Z) < \dim X = |a|$. 
Since $L^a$ is a general linear space then,  $(Z \cap X) \cap L^a = \emptyset$.  
Therefore, we have that \[
  (Z\cap L^a\cap W)-(X\cap Z)=Z\cap L^a\cap W 
\]
and in particular the degrees match, proving the first equality.
Additionally, $L^a$ and $W$ are general, so the intersection is transverse and we have 
\[ 
[Z\cap L^a\cap W]=[Z] \cdot h^{a} \cdot \alpha^{N-|a|}.
\] 
Substituting this into \eqref{eq:Lambda} gives $\Lambda_a(X \cap Z, Z) = 0$.
\end{proof}

\begin{proposition} \label{corr:LambdaDimXReducible}
  Let $X$ and $Z$ be closed subschemes of $\pp^{\mathbf{n}}_m$ and let $Z_1,\dots,Z_t$ be the primary components of $Z$ with $Z_i\not \subset X$ for all $i$. Suppose that $X_{\rm red}\subset Z_i$ for $i\leq \rho$, and $X_{\rm red}\not \subset Z_i$ for $i>\rho$. Then 
\[
  \{\Lambda(X \cap Z, Z)\}_{\dim(X)}= \{\Lambda(X \cap Z,Z_1 \cup \cdots \cup Z_\rho)\}_{\dim(X)}.
\]
Hence, if $X \subset Z$ the $\dim(X)$ part of $s(X,Z)$ is
    \[
\{s(X,Z)\}_{\dim(X)}= \left\lbrace \alpha^{\dim(Z)-\dim(X)}[{Z}] - \sum_{|a|=\dim(X)} g_a(X,{Z}) \cdot [\pp^a_m]  \right\rbrace_{\dim(X)}.
\] \label{cor:SegreDimX} 
\end{proposition}
\begin{proof}
Fix $a\in \ZZ_{>0}^m$ such that $|a|= \dim(X)$. Again we have
\[ g_a(X \cap Z,Z) = \deg((Z \cap L^a \cap W) - (X \cap Z)). \]
This can be expanded as
\begin{align*}
g_a(X \cap Z,Z) &= \deg(Z \cap L^a \cap W - (X \cap Z))\\
&= \sum_{i=1}^t \;  \deg(Z_i \cap L^a \cap W - (X \cap Z))
\end{align*}
where the second equality follows from the fact that $L^a$ and $W$ are general, so the finite set
$Z_i \cap L^a \cap W - (X\cap Z)$ is disjoint from $Z_j \cap L^a \cap W - (X \cap Z)$ when $i\neq j$. 
Since none of $Z_{\rho+1},\dots, Z_t$ contain $X$, by Lemma \ref{propn:reducible_ProjDeg} we have that 
\[
\sum_{i=\rho+1}^t \;  \deg(Z_i \cap L^a \cap W) = \sum_{i=\rho+1}^t \;  \int L^a \cdot \alpha^{N-|a|} [Z_i].
\]By definition,
\[\{\Lambda({X \cap Z,Z})\}_{\dim(X)} =  \alpha^{N-\dim(X)}[Z] - \sum_{|a|=\dim(X)} g_{a}(X \cap Z,Z)
  [\pp^a_m]\]
and the first term is $\sum_i \alpha^{N-\dim(X)}[Z_i]$, we have shown that the summands corresponding to
$Z_{\rho+1},\dots,Z_t$ cancel, and so
\begin{align*}
  \{\Lambda({X \cap Z,Z})\}_{\dim(X)} &=  \sum_{i=1}^t \; \alpha^{N-\dim(X)}[Z_i] - \sum_{|a|=\dim(X)} g_{a}(X \cap Z,Z) [\pp^a_m]\\
  &=  \sum_{i=1}^\rho \; \alpha^{N-\dim(X)}[Z_i] - \sum_{|a|=\dim(X)} g_{a}(X \cap \hat{Z},\hat{Z}) [\pp^a_m]\\
  &=\{\Lambda({X \cap \hat{Z}, \hat{Z}})\}_{\dim(X)},
\end{align*}
where $\hat{Z} = Z_1 \cup \cdots \cup Z_\rho$.

Now suppose $X \subset Z$. Then $Z = \hat Z$ and the proof of the recursive formula of Theorem \ref{MainTheorem2MultiProj} (which begins in dimension $\dim(X)$), applied to this case, gives us that the dim($X$) part of $s(X,Z)$ is 
\[
\{s(X,Z)\}_{\dim(X)}=\alpha^{\dim(Z)-\dim(X)}[Z] - \sum_{|a|=\dim(X)} g_a(X,Z) \cdot[\pp^a_m]. \qedhere
\]
\end{proof}

Note that, in particular, Proposition \ref{cor:SegreDimX} means that even if $Z$ is a reducible subscheme of $\pp^{\bf n}_m$ we can compute $\{s(X,Z)\}_{\dim(X)}$ using Theorem \ref{MainTheorem1_Multi_proj} directly, \textit{without} knowing the irreducible components of $Z$. Further, even if $Z$ is an arbitrary (possibly reducible) subscheme of $\pp^{\bf n}_m$ the result of Theorem \ref{MainTheorem2MultiProj} above can still be used to compute the entire Segre class $s(X,Z)$, provided we can compute the irreducible components of $Z$. In particular, if $Z$ has irreducible (and reduced) components $Z_1,\dots, Z_s$ where $Z_i$ has geometric multiplicity $\mathfrak{m}_i$, then  \cite[Lemma~4.2]{fulton2013intersection} gives $$
  s(X,Z)=\sum_{i=1}^s\mathfrak{m}_is(X\cap Z_i,Z_i).
  $$   
\subsection{Computing Segre classes in $A^*(T_\Sigma)$}\label{subsec:SegSubToric}
In this subsection we generalize the results of \S\ref{subSec:SegreMultiProj} to a subscheme of a smooth projective toric variety $T_\Sigma$ defined by a fan $\Sigma$. 
While the majority of the results from the previous section carry over mutatis mutandis, a few details should be discussed.
The main purpose of this section is to record these differences.
The Chow ring $A^*(T_\Sigma)$ of $T_\Sigma$ is explicitly determined by the fan $\Sigma$; see \cite[\S12.5]{cox2011toric} for a discussion of this. A Cartier divisor $D$
on $T_\Sigma$ is termed \textit{numerically effective} or \textit{nef} if
$D\cdot C \geq 0$ for every irreducible complete curve $C\subset
T_\Sigma$.

\begin{remark}
Proposition 6.3.24 of \cite{cox2011toric} tells us that when
$T_{\Sigma}$ is a smooth projective toric variety there exists a basis for
$A^1(T_{\Sigma})$ consisting of nef divisors.
Each divisor in this basis is base-point free since by Theorem 6.3.12
of \cite{cox2011toric} a Cartier divisor $D$ on $T_\Sigma$ is nef
if and only if $D$ is base-point free. \label{remark:toricBasepointFree}
\end{remark}

\subsubsection{Segre class computation for subschemes of $T_{\Sigma}$}
Remark \ref{remark:toricBasepointFree} allows us to generalize many results in \S\ref{subSec:SegreMultiProj} from the ambient space $\pp^{n_1}\times \cdots\times \pp^{n_m}$ to any smooth projective toric variety in a fairly straightforward fashion. The main practical difference is that an analogue to the simple dehomogenization method used in Theorem \ref{MainTheorem1_Multi_proj} may not exist. See \cite[\S3.1]{helmer2017toric} for a discussion of this. Below we state the results from \S\ref{subSec:SegreMultiProj} that continue to hold in this more general setting. We will again employ the notation of \S\ref{subsec:ProjectiveDegrees} for the projective degrees. 

Let $b_1,\dots, b_{m} \in A^1(T_{\Sigma})$ be a fixed nef basis for $A^1(T_{\Sigma})$. 
We express the rational equivalence class of a point in $T_\Sigma$ as the monomial $b^n=b_1^{n_1}\cdots b_{m}^{n_m}$.  
The monomials $b^e$ which divide $b^n$ in $A^*(T_\Sigma)$ form a basis for $A^{|e|}(T_\Sigma)$. 
\begin{convention}
Let $T_\Sigma$ be a smooth projective toric variety. For pairs $X\subset Y \subset T_\Sigma$ we will frequently use the following conventions:
\begin{enumerate}
\item $Y \subset T_\Sigma$ is an irreducible subscheme of dimension $N = \dim(Y)$,
\item $X \subset Y$ is a closed subscheme defined by $\alpha$-homogeneous polynomials $f_0,\dots,f_r$ in the Cox ring $R$ of $T_\Sigma$.
\end{enumerate}\label{conv:X_Y_Toric}
\end{convention}
Recall that the $\alpha$-homogeneity assumption on the defining equations of $X$ is made without loss of generality; see \S\ref{subsubsec:HomGens} and Definition \ref{def:alpha-homogeneous}.
\begin{proposition}
The projective degrees of $X$ in $Y$ are given by
\[ g_{a}(X,Y) = \deg \left( \left(Y \cap L^a \cap W \right) - X\right)\] where $[L^a] = b^a$ and $
W=\mathbb{V}(P_1,\dots,P_{\dim(Y)-|a|})$, with $ P_j=\sum_i \lambda_i f_i
$
for general $\lambda_i\in k$. 
\label{propn:projective_degreeToric}
\end{proposition}\begin{proof}
The divisors $b_i$ are nef, meaning the corresponding line bundles are generated by global sections and Kleiman's transversality theorem \cite{kleiman1974transversality} applies. Hence, the proof is entirely analogous to that of Proposition \ref{propn:projective_degreeFirstMultiProj} with the $b_i$ taking the place of the $h_i$.
\end{proof}
\begin{remark}
To apply Proposition \ref{propn:projective_degreeToric} above in practice, one must first explicitly compute the saturation $J=(\mathcal{I}_Y+\mathcal{I}_{L^a}+\mathcal{I}_W):\mathcal{I}_X^\infty$ in the multi-graded Cox ring $R$. This step is often quite computationally expensive, and is avoided in Theorem \ref{MainTheorem1_Multi_proj} because we can always dehomogenize. Having computed $J$, one must then employ some method to count (with multiplicities) the points in the scheme $S$ defined by ${J}$. In particular, as noted above, there is in general no analogue of Theorem \ref{MainTheorem1_Multi_proj} in this setting, which limits the methods that may be applied for this later computation as well. In \cite[Theorem~3.1]{helmer2017toric} a combinatorial criterion for the existence of an analogue to Theorem \ref{MainTheorem1_Multi_proj} is given, however this criterion is quite restrictive. One broadly applicable approach is to use symbolic methods based on results of \cite{miller2004combinatorial} such as the \texttt{Macaulay2} \cite{M2} command \texttt{multidegree} to count the points in $S$. A numeric approach using, for example, homotopy continuation and Cox's geometric quotient construction \cite[\S5]{cox2011toric} is theoretically possible, however to the best of our knowledge no such implementation exists. In practice this means that computations involving subschemes of an arbitrary smooth projective toric variety $T_\Sigma$ (to yield Segre classes in $A^*(T_\Sigma)$) are possible, but tend to be more computationally expensive and have less flexibility in the underlying computational methods we can employ to implement them.   \end{remark}
\begin{theorem}

Write the Segre class as $$s(X,Y) \;\;=\hspace{-1mm}\sum_{|i|\leq \dim(X)}  s_{i}(X,Y) b^{n-i}.$$ With $\Lambda(X,Y)$ as in \eqref{eq:Lambda} we have the following recursive formula for $s(X,Y)$:
\[
s_{a}(X,Y)=\Lambda_{a}(X,Y)-\int (1+\alpha)^{N-|a|} b^a \sum_{{|i|>|a|} }s_{i}(X,Y) b^{n-i} \;.
\] \label{thm:MainSegreToric}
\end{theorem}
\begin{proof}
The proof is analogous to that of Theorem \ref{MainTheorem2MultiProj}, with $b_i$ replacing $h_i$.
\end{proof}
 \begin{corollary}
  Let $X$ and $Z$ be closed subschemes of $T_\Sigma$ and let $Z_1,\dots,Z_t$ be the primary components of $Z$ with $Z_i\not \subset X$ for all $i$. Suppose that $X_{\rm red}\subset Z_i$ for $i\leq \rho$, and $X_{\rm red}\not \subset Z_i$ for $i>\rho$. Then 
\[
  \{\Lambda(X \cap Z, Z)\}_{\dim(X)}= \{\Lambda(X \cap Z,Z_1 \cup \cdots \cup Z_\rho)\}_{\dim(X)}.
\]
Hence, if $X \subset Z$ the $\dim(X)$ part of $s(X,Z)$ is
    \[
\{s(X,Z)\}_{\dim(X)}= \left\{\alpha^{\dim(Z)-\dim(X)}[{Z}] - \sum_{|a|=\dim(X)} g_a(X,{Z}) \cdot b^{n-a}\right\}_{\dim(X)} .
\] \label{cor:SegreDimXToric}

\end{corollary} 
\begin{proof} The result corresponding to Lemma \ref{propn:reducible_ProjDeg} holds in the Chow ring $A^*(T_\Sigma)$ where the projective degrees are computed using Proposition \ref{propn:projective_degreeToric}. Using this result the remainder of the proof is identical to that of Proposition \ref{cor:SegreDimX}.
  \end{proof}

 \begin{remark}
Remark \ref{remark:GeneralChoiceOFLinSpace} generalizes immediately to the case where $\pp^{n_1}\times\cdots\times\pp^{n_m}$ is replaced by a smooth projective toric variety $ T_\Sigma$ with fixed nef basis.
 \end{remark}

\section{Intersection theory}
\label{sec:intersection-theory}
In this section we employ the results from \S\ref{sec:MainResults} to give a method to compute the intersection product of two varieties inside another (for more on intersection products see \cite[\S6]{fulton2013intersection}).
Fix a smooth projective toric variety $T_\Sigma$.
Let $Y \subset T_\Sigma$ be a nonsingular variety
with $X \subset Y$ a regular embedding and $V \subset Y$ any subvariety.
The standard setup for an intersection product is the fiber square
\begin{equation}
 \begin{tikzcd}
  X \cap V \ar[r,"e_1"] \ar[d,"e_2"] & V \ar[d]\\
  X \ar[r] & |[alias=Y]| Y \rlap{\ .}
 \end{tikzcd}
 \label{fig:std-int-square}
\end{equation}

We would like to compute the intersection product $X \cdot_Y V \in A_*(X \cap V)$. This is hopeless in
the generality stated since we can't a priori know anything about the Chow ring of an arbitrary variety. However, we
could learn the pushforward of the intersection product to a class in $A^*(T_\Sigma)$. In this computation there are two ingredients: the total Chern class $c(N_X Y)$ of the normal bundle to $X$ in $Y$, and the Segre class $s(X \cap
 V,V)$. The intersection product \cite[Proposition 6.1]{fulton2013intersection},
 denoted $X \cdot_Y V$, is the part of the class $c(e_2^*N_X Y) \frown s(X \cap V,
 V) $ in the \emph{expected dimension} $\dim(V)-(\dim(Y)-\dim(X))$, that is
\[ X \cdot_Y V = \left\{ c(e_2^*N_X Y) \frown s(X \cap V, V) \right\}_{exp.~dim.} \in A_*(X \cap V). \]
If we know both pushforwards to the ambient projective space, it is possible to perform the computation there. However, often the Chern class $c(N_X Y)$ is difficult (or impossible) to obtain in practice. 
We can get around this by a reduction to the diagonal:
\begin{equation}
  \begin{tikzcd}
  X \cap V \ar[r] \ar[d,"i"] & X \times V \ar[d]\\
  Y \ar[r] \ar[d,"j"] & Y \times Y \ar[d] \\
T_\Sigma \ar[r,"\Delta"] & |[alias=T]| T_\Sigma \times T_\Sigma \rlap{\ .}
 \end{tikzcd} \label{fig:reduction-diag}
\end{equation}

The intersection product corresponding to the upper square in \eqref{fig:reduction-diag} is the same class $X \cdot_Y V$ in $A_*(X \cap V)$ as that associated to \eqref{fig:std-int-square}.
For this intersection product the required Segre class is $s(X \cap V, X \times V)$ and the Chern class is simply the Chern class $c(TY)$ of the tangent bundle. Theorem \ref{thm:intersectionProduct} allows us to compute the pushforward to $A^*(T_\Sigma)$ of the intersection product $X \cdot_Y V$. We use the notation from \eqref{fig:reduction-diag} in the statement of Theorem \ref{thm:intersectionProduct} below.

\begin{theorem}\label{thm:intersectionProduct}
  Suppose $TY$ is the pullback $j^*E$ of some vector bundle on $T_\Sigma$.
 Then the class $X \cdot_Y V$, pushed forward to $T_\Sigma$, is given
 by \[
   j_* i_* (X \cdot_Y V) = \left\{ c(E) \frown \Delta^*\underline{s}(X
     \cap V, X \times V) \right\}_{exp.~dim.}\in A^*(T_\Sigma),
 \]
 where
 $\underline{s}(X \cap V, X \times V)$ denotes the pushforward to $T_\Sigma \times T_\Sigma$.
 \end{theorem}
 \begin{proof}

  The intersection product $X \cdot_Y V$ is given by $c(i^*TY) \frown s(X \cap V, X \times V)$.
  Then we have the following, via successive application of the projection formula:
  \begin{align*}
   j_* i_* \left( c(i^*TY) \frown s(X \cap V, X \times V) \right) & = j_* \left( c(TY) \frown i_* s(X \cap V, X \times V) \right) \\
   & = c(E) \frown j_* i_* s(X \cap V, X \times V) \\
   & = c(E) \frown \Delta^* \underline{s}(X\cap V, X \times V)
  \end{align*}
  where the last step follows from the fact that the pushforward from $X \cap V$ to $T_\Sigma \times T_\Sigma$ factors through $\Delta$.
\end{proof}

We note that the bundle $E$ in Theorem \ref{thm:intersectionProduct} above may not exist. Further, is not clear how to find the bundle $E$ even when its existence is known.
As such, the following variation of
Theorem \ref{thm:intersectionProduct} is more computationally useful; 
 we again use the notation of \eqref{fig:reduction-diag}. 
 
\begin{theorem}
Suppose that the fan $\Sigma$ has rays $\{\rho_1, \dots, \rho_m\}$ and let $D_{\rho_i}$ be the divsor associated to $\rho_i$ in $A^1(T_\Sigma)$. Assume $Y \subset T_\Sigma$ is a smooth complete intersection of codimension $r$ defined by $f_1,\dots, f_r$ with $[V(f_i)]=\alpha_i\in A^*(T_{\Sigma})$. Then
  \[ X \cdot_Y V = \left\{ \frac{\prod_{i=1}^m(1+D_{\rho_i})}{\prod_{i=1}^r (1+\alpha_i)} \frown
    \Delta^* \underline{s}(X \cap V, X \times V) \right\}_{exp.~dim.} \in A^*(T_\Sigma) .\]
  \label{thm:intersectionProductCI}
\end{theorem}
\begin{proof}
  If $Y$ is a smooth complete intersection, then its tangent bundle is
  determined by the exact sequence
  \[
    0 \to TY \to T(T_\Sigma|_Y) \to N_Y T_\Sigma \to 0.
  \]
The Whitney sum formula gives a description of its Chern classes:
  \[ c(TY) = \frac{c(T(T_\Sigma|_Y))}{c(N_Y T_\Sigma)} = \frac{c(j^* T(T_\Sigma))}{c(j^* F)}\]
  where $F = \bigoplus \mathcal{O}(\alpha_i)$ is a sum of line bundles. By \cite[Proposition~13.1.2]{cox2011toric} we have that $c(T(T_\Sigma))=\prod_{i=1}^m(1+D_{\rho_i}) \in A^*(T_\Sigma)$.
  Combining this with Theorem \ref{thm:intersectionProduct} we have
  \begin{align*}
   j_* i_* \left( c(i^*TY) \frown s(X \cap V, X \times V) \right) & = j_* \left( c(TY) \frown i_* s(X \cap V, X \times V) \right) \\
   & = j_* \frac{c(j^* T(T_\Sigma))}{c(j^*F)} \frown i_* s(X \cap V, X \times V) \\
   & = \frac{c(T(T_\Sigma))}{c(\bigoplus \mathcal{O}(\alpha_i))} \frown j_* i_* s(X \cap V, X \times V) \\
   & = \frac{\prod_{i=1}^m(1+D_{\rho_i})}{\prod_{i=1}^r (1+\alpha_i)} \frown j_* i_* s(X \cap V, X \times V) \\
   & = \frac{\prod_{i=1}^m(1+D_{\rho_i})}{\prod_{i=1}^r (1+\alpha_i)} \frown \Delta^* \underline{s}(X \cap V, X \times V). \qedhere
  \end{align*}
\end{proof}

In the case where $T_\Sigma = \pp^n$, 
Theorem \ref{thm:intersectionProductCI} simplifies to the following. 

\begin{corollary}
  Assume $Y \subset \pp^n$ is a smooth complete intersection. Then
  \[ X \cdot_Y V = \left\{ \frac{(1+{h})^{n+1}}{\prod_{i=1}^r (1+d_i {h})} \frown
    \Delta^* \underline{s}(X \cap V, X \times V) \right\}_{exp.~dim.} \in A^*(\pp^n)\cong \ZZ[h]/\langle h^{n+1} \rangle,\]
  where $Y$ is defined by $r=\codim(Y)$ polynomials of degrees $d_1,\dots,d_r$.
  \label{thm:intersectionProductCIPPn}
\end{corollary}

\section{Algebraic multiplicity}\label{subSection:AlgMult1}
In this section, we consider applications of the results from \S\ref{sec:MainResults} to the computation of Samuel's algebraic multiplicity of an ideal in a local ring. In particular, we prove a new explicit expression for these multiplicities which does not require working in local rings. We note that our results in this section can be seen as a generalization of classical techniques such as those presented in \cite[pg.~259]{harris1992algebraic}, which give geometric expressions for the algebraic multiplicity of a point in a projective variety. 

Let $R$ be the Cox ring of a smooth projective toric variety $T_\Sigma$. Let $\mathcal{I}_X$ be a prime ideal in $R$ and $\mathcal{I}_X\supset \mathcal{I}_Y$ be a primary ideal in $R$, defining an irreducible scheme $Y$ and a subvariety $X=\mathbb{V}(\mathcal{I}_X)\subset Y$. The local ring of $Y$ along $X$ is defined as the localization of $R/\mathcal{I}_Y$ at the prime ideal $\mathcal{I}_X$, that is $$
\mathcal{O}_{X,Y}=(R/\mathcal{I}_Y)_{\mathcal{I}_X}.
$$Let $\mathcal{M}$ denote the maximal ideal of $\mathcal{O}_{X,Y}$. For $t>>0$ and $d=\mathrm{codim}(X,Y)$ the \textit{Hilbert-Samuel polynomial} is$$
P_{HS}(t):=\ell(\mathcal{O}_{X,Y}/\mathcal{M}^t)={e_XY}\cdot \frac{ t^{d}}{d!} +\mathrm{ lower \; terms}.
$$ The coefficient {$e_XY$} of the leading term of the Hilbert-Samuel polynomial is known as the {\textit{algebraic multiplicity} of $Y$ along $X$}. The coefficient $e_XY$ is also the multiplicity of the ideal $\mathcal{I}_X\cdot \mathcal{O}_{X,Y}$ in the local ring $ \mathcal{O}_{X,Y}$, denoted $e(\mathcal{I}_X,\mathcal{O}_{X,Y})$. This definition is due to Samuel \cite{samuel1955methodes}. 
Samuel's multiplicity {$e_XY$} is also given by the {integer coefficient of $[X]$ in $s(X,Y)$}; see \cite[Ex.~4.3.4]{fulton2013intersection}. We take this later characterization of $e_XY$ as its definition for the remainder of this note. 
\begin{definition}
Let $Y$ be a pure-dimensional subscheme of a smooth projective toric variety $T_\Sigma$ and let $X$ be a subvariety of $Y$. The {algebraic multiplicity} of $Y$ along $X$, denoted $e_XY$, is the integer coefficient of $[X]$ in the class $s(X,Y) \in A^*(T_\Sigma)$.\label{def:eXY}
\end{definition}

Using Proposition \ref{cor:SegreDimX} we obtain the following expression for the algebraic multiplicity in $T_\Sigma$. 
We use the notation
of \S\ref{subsec:notationConvention}.
As in \S\ref{subsec:SegSubToric} we let $b_1,\dots, b_{m} \in A^1(T_{\Sigma})$ be a fixed nef basis for $A^1(T_{\Sigma})$. 
We express the rational equivalence class of a point in $T_\Sigma$ as the monomial $b^n=b_1^{n_1}\cdots b_{m}^{n_m}$.  
The monomials $b^e$ which divide $b^n$ in $A^*(T_\Sigma)$ form a basis for $A^{|e|}(T_\Sigma)$. 

\begin{theorem}
Let $Y \subset T_\Sigma$ be a pure-dimensional subscheme, and let $V \subset Y$ be a non-empty subvariety defined by $\alpha$-homogeneous polynomials. Write
\[
[V]=\sum_{|a|=\dim(V)} v_a b^{n-a}, \quad\text{and}\quad \alpha^{\dim(Y)-\dim(V)}[Y]=\sum_{|a|=\dim(V)}{y}_a b^{n-a}.
\]
Let $g_a(V,Y)$ be the projective degree as in \eqref{eq:ProjDegsExplicit_No_G}. 
Then the algebraic multiplicity of $Y$ along $V$ is given by 
\[
e_VY=\frac{y_a-g_{a}(V,Y)}{v_a} \in \ZZ_{>0}.
\]
for any $a$ such that $|a|=\dim(V)$ and $v_a\neq 0$.
\label{thm:eX_Y_Toric}
\label{thm:eX_Y_MultiProj}
\end{theorem}\begin{proof}
 From Definition \ref{def:eXY} we have that $e_VY$ is the integer coefficient of $[V]$ in $s(V,Y)$. From the definition of the Segre class (Definition \ref{def:Segre}) we see that the $\dim(V)$ part of $s(V,Y)\in A^*(T_\Sigma)$ is an integer multiple of $[V]$ in the same Chow ring. Since $V$ is irreducible, the polynomial in $A^*(T_\Sigma)$ representing $\{s(V,Y)\}_{\dim(V)}$ is an integer multiple of the polynomial representing $[V]$. By Proposition \ref{cor:SegreDimX}, we have that
\[ 
\{s(V,Y)\}_{\dim(V)}=\alpha^{\dim(Y)-\dim(V)}[Y]-\sum_{|a|=\dim(V)} g_a(V,Y) b^{n-a},
\] 
so $e_VY=\frac{y_a-g_{a}(V,Y)}{v_a}$ for any $a$ for which $v_a$ is non-zero (at least one such $a$ must exist since we assume $V$ is non-empty). 
\end{proof}

In the case where $T_\Sigma=\PP^{n_1}\times \cdots \times \PP^{n_m}$ the Chow ring is represented as in \eqref{eq:ChowRingMultiProj}, that is $A^*(\PP^{n_1}\times \cdots \times \PP^{n_m})\cong \ZZ[h_1,\dots, h_m]/\langle h_1^{n_1+1}, \dots,h_m^{n_m+1}\rangle $, hence in Theorem \ref{thm:eX_Y_Toric} we simply replace $b^n$ with $h^n=h_1^{n_1}\cdots h_m^{n_m}$. In the case where $T_\Sigma$ is a single projective space this result has a particularily appealing form. 
\begin{theorem}
Let $Y \subset \pp^n$ be a pure-dimensional subscheme and $X \subset Y$ a non-empty subvariety. Let $d$ be the maximum degree of the generators of $\mathcal{I}_X$ and $\mathcal{I}_Y$.
Then the algebraic multiplicity of $Y$ along $X$ is
\[
e_XY=\frac{\deg(Y)d^{\dim(Y)-\dim(X)}-g_{\dim(X)}(X,Y)}{\deg(X)} \in \ZZ_{>0}.
\] \label{theorem:eXY_PPn}
\end{theorem}
\begin{proof}
In this case the Chow ring is $A^*(\pp^n)\cong \ZZ[h]/\langle h^{n+1} \rangle$. The proof goes through as in the previous result (with $b^n$ replaced by $h^n)$, but in this case the expression for the dimension-$X$ part of $s(X,Y)$ simplifies to
\[ 
\{s(X,Y)\}_{\dim(X)}=\left(\deg(Y)d^{\dim(Y)-\dim(X)}-g_{\dim(X)}(X,Y)\right)h^{n-dim(X)}. 
\]Also note that $[X]=\deg(X)h^{n-\dim(X)}$. The conclusion follows. 
\end{proof}

\section{Gr\"obner-free containment testing for ideals and varieties}\label{sec:GBFreeContainment}
In this section we develop a new algorithm to test if $V\subset W$ for $V,W$ (possibly singular) subvarieties of a smooth projective toric variety $T_\Sigma$. The key step in this algorithm is counting the number of solutions to a zero-dimensional system of polynomial equations, for which there are many Gr\"obner-free symbolic and numeric methods (e.g.~geometric resolutions \cite{lecerf2003computing,Kronecker}, homotopy continuation \cite{hauenstein2018solving,verschelde1999algorithm,BHSW06}, etc.). As in previous sections we freely use the notations and conventions of \S\ref{subsec:notationConvention}.

\subsection{Containment of a subvariety in the singular locus of a variety}\label{subsec:subVarInSingLocus}
In this subsection we employ a result of Samuel \cite{samuel1955methodes} (see also Fulton \cite[Ex.~12.4.5(b)]{fulton2013intersection}) which relates containment of varieties to the algebraic multiplicity studied in \S\ref{subSection:AlgMult1}. This yields an algorithm to test containment of a subvariety in the singular locus of a variety \textit{without computing the singular locus}. Samuel proves the following:
\begin{proposition}[\cite{samuel1955methodes} II \S6.2b]
Let $Z$ be a subvariety of a smooth projective toric variety $T_{\Sigma}$ and suppose that $X$ is a subvariety of $Z$. Then $e_XZ=1$ if and only if $X$ is \textit{not contained in the singular locus of $Z$}. \label{prop:SamualMultSingLocusCont}
\end{proposition}
Combining Proposition \ref{prop:SamualMultSingLocusCont} with the result of Theorem \ref{theorem:eXY_PPn} gives a particularly simple and explicit numerical test of containment in the singular locus when considering a pair of projective varieties. No information about the ideal defining the singular locus is used. 
\begin{corollary}
Let $Z$ be a subvariety of $\PP^n$ and suppose that $X$ is a subvariety of $Z$. Let $d$ be the maximum degree of the generators of $\mathcal{I}_X$ and $\mathcal{I}_Z$. We have that $X$ is \textit{not contained in the singular locus of $Z$} if and only if 
\[
\frac{\deg(Z)d^{\dim(Z)-\dim(X)}-g_{\dim(X)}(X,Z)}{\deg(X)}=1.\label{cor:contSingLocusPn}
\]
\end{corollary}
\begin{proof}
This follows immediately from Proposition \ref{prop:SamualMultSingLocusCont} using the expression for $e_XZ$ given in Theorem \ref{theorem:eXY_PPn}.
\end{proof}
Analogous statements can be written by combining Proposition \ref{prop:SamualMultSingLocusCont} and Theorem \ref{thm:eX_Y_Toric} to obtain a numerical criterion for the containment of a variety $X$ in the singular locus of a subvariety $Y$ of some smooth projective toric variety $T_\Sigma$. In Proposition \ref{propn:SingLocusSubschemeZ} below we also remove the requirement that $X$ is irreducible. As in \S\ref{subsec:SegSubToric} let $b_1,\dots, b_{m} \in A^1(T_{\Sigma})$ be a fixed nef basis for $A^1(T_{\Sigma})$.

\begin{proposition}
Let $Z \subset T_\Sigma$ be an irreducible subscheme with geometric multiplicity $\mathfrak{m}_Z$. Let $X$ be a reduced subscheme of $Z$. Choose an $\alpha$-homogeneous system of generators for $X$ and write \[
[X]=\sum_{|a|=\dim(X)}q_ab^{n-a}, \quad \text{and} \quad \alpha^{\dim(Z)-\dim(X)}[Z]=\sum_{|a|=\dim(X)}z_ab^{n-a}.
\]
Then an irreducible component $V$ of $X$ having $\dim(V)=dim(X)$ is contained in the singular locus of $Z$ if and only if \[
{z_a-g_{a}(X,Z)}>\mathfrak{m}_Z\cdot {q_a} 
\]for some $a$ such that $|a|=\dim(X)$ and $q_a\neq 0$.\label{propn:SingLocusSubschemeZ}

\end{proposition}\begin{proof}
 
By \cite[Lemma 4.2]{fulton2013intersection}, we have that $\{s(X,Z)\}_{\dim(X)}=\mathfrak{m}_Z\cdot \{s(X,Z_{\rm red})\}_{\dim(X)}$. Let $X_1,\dots, X_\ell$ be the irreducible components of $X$. From \cite[Example~4.3.4]{fulton2013intersection} we have that 
\begin{align*}
\{s(X,Z_{\rm red})\}_{\dim(X)}&= \sum_{i=1}^l e_{X_i}(Z_{\rm red})[X_i].
\end{align*}
By Proposition \ref{prop:SamualMultSingLocusCont} we have that $e_{X_i}Z_{\rm red}=1$ if and only if $X_i$ is not contained in the singular locus of $Z_{\rm red}$. It follows that 
\[
\{s(X,Z)\}_{\dim(X)} = \mathfrak{m}_Z\cdot \{s(X,Z_{\rm red})\}_{\dim(X)}=\mathfrak{m}_Z\cdot([X_1]+\cdots+[X_\ell])=\mathfrak{m}_Z\cdot[X]
\]
if and only if no $X_i$ is contained in the singular locus of $Z$. Applying Corollary \ref{cor:SegreDimXToric} gives the conclusion. 
\end{proof}
\subsection{Containment of any two varieties}\label{subsec:ContainmentTest}
In this subsection we adapt the method presented in \S\ref{subsec:subVarInSingLocus} to test containment of varieties in general. Unlike in \S\ref{subsec:subVarInSingLocus}, where we were able to study simpler objects than standard methods, this will lead us to construct ideals which may be more complicated to perform the test. However, the methods presented in this section do not require the computation of a Gr\"obner basis and in particular will give a means to test containment of possibly singular varieties using numerical algebraic geometry. 

In the lemma below $T_\Sigma$ again denotes a smooth projective toric variety with Cox ring $R$.
\begin{lemma}
Let $X \subset T_\Sigma$ be a variety defined by an $\alpha$-homogeneous ideal 
$\mathcal{I}_X= \langle f_0,\dots,f_r \rangle$.
Then
\[\{s(X,\Theta)\}_{\dim(X)}=[X]\]
where $\Theta$ is a hypersurface defined by a general $k$-linear combination $\sum \lambda_{j} f_j$.
\label{lemma:dimXSegreisXforIrVarInRandomPlane}
\end{lemma}
\begin{proof}
Let $W_c$ be the scheme defined by general $k$-linear combinations $P_1,\dots, P_c$ as in \eqref{eq:Pjdef}.
Then using the equality in Proposition \ref{propn:projective_degreeToric}, and writing $P_{\dim(\Theta)-|a|+1}=\sum \lambda_{j} f_j$, we have
\begin{align*}
g_a(X,\Theta) &= \deg({\Theta \cap L^a \cap W_{\dim(\Theta)-|a|} - X})\\
&= \deg({L^a \cap W_{\dim(\Theta)-|a|+1} - X})\\
&= \deg({L^a \cap W_{\dim(T_\Sigma)-|a|} - X})\\
&= g_a(X,T_\Sigma).
\end{align*}That is, the projective degrees of $X$ in $\Theta$ are the same as those of $X$ in the ambient space $T_{\Sigma}$. 
By Corollary \ref{cor:SegreDimXToric}, this implies $\{s(X,\Theta)\}_{\dim(X)} = \{s(X,T_\Sigma)\}_{\dim(X)}$, but since $X$ is a variety, $\{s(X,T_\Sigma)\}_{\dim(X)} = [X]$ by Proposition \ref{prop:SamualMultSingLocusCont}.
\end{proof}

\begin{remark}
The result above does not hold for $X$ an arbitrary scheme.  For example, let $X$ be the point-scheme defined by $\langle x^3,xy,y^3 \rangle$ in $\pp^2$.  Then $[X]=5h^2$, but $s(X,\Theta)=6h^2$.
\end{remark}

\begin{theorem}
Let $X=\mathbb{V}(f_0,\dots, f_r)$ be a subvariety and $Y=\mathbb{V}(\omega_1,\dots,\omega_s)$ a reduced subscheme of
$T_\Sigma$ where the given defining polynomials have the same multidegree, and let $\dim(X)\leq \dim(Y)$. 
Let $\sum \lambda_{i} f_i$ and $\sum \gamma_{j} \omega_j$ be general $k$-linear combinations defining 
hypersurfaces $\Theta$ and $\Omega$, respectively.
If $Z:= \Theta \cup \Omega$, then $X\subset Y$ if and only if $e_XZ> 1$.\label{thm:containmentGeneralIrrVar}
\end{theorem}
\begin{proof}
For the forward direction, we know that $e_XZ$ is a sum taken over the irreducible components of $Z$ which contain $X$, and by Theorem \ref{thm:eX_Y_MultiProj} each summand must be positive. 
If $X \subset Y$, then $X \subset \Omega$, and by construction $X \subset \Theta$.
Then $e_X \Omega \geq 1$ and by Lemma \ref{lemma:dimXSegreisXforIrVarInRandomPlane}, $ e_X \Theta=1$, so
\[e_XZ = e_X \Omega+ e_X \Theta > 1.\] 

Now suppose that $e_XZ > 1$.
Then, since $e_X\Theta = 1$, we must have another component containing $X$.  That is, $X \subset \Omega$.
Since $\Omega$ is defined by a general $k$-linear combination of the $\omega_i$'s, each $\omega_i$ must vanish on $X$ and
we can conclude that $X \subset Y$.
\end{proof}

\begin{example}[Containment testing for subvarieties of a product of projective spaces] Work in $\pp_x^2\times\pp_y^2\times\pp_z^2$ with coordinate ring $R=k[x_0,x_1,x_2,y_0,y_1,y_2,z_0,z_1,z_2]$ graded by $A^1(\pp_x^2\times\pp_y^2\times\pp_z^2)$ with generators $h_x$, $h_y$, $h_z$. The irrelevant ideal is $B=\langle x_0,x_1,x_2\rangle \cdot \langle y_0,y_1,y_2 \rangle \cdot \langle z_0,z_1,z_2\rangle$. Let $F$ be a general polynomial of multidegree $h_x+h_y+h_z$ in $R$ and let \begin{align*}
&f_1=- 10x_2y_1z_0 + 2x_1y_2z_0 + 35x_2y_0z_1 - 7x_0y_2z_1 - 25x_1y_0z_2 + 25x_0y_1z_2\\
&f_2=9x_2y_1z_0 - 9x_1y_2z_0 - 4x_2y_0z_1 + 4x_0y_2z_1 + 3x_1y_0z_2 - 3x_0y_1z_2.
\end{align*}
Consider the varieties $X$ and $Y$ in $\pp_x^2 \times \pp_y^2 \times \pp_z^2$ defined by 
\begin{align*}
\mathcal{I}_X&= \langle y_0,y_1,y_2 \rangle \cdot \langle f_1, f_2 \rangle + \langle F \rangle, \\
\mathcal{I}_Y&=\left\langle z_0\cdot f_1-z_1\cdot f_2,F\right\rangle.
\end{align*} 
To test if $X$ is contained in $Y$ using Theorem \ref{thm:containmentGeneralIrrVar}, we define a general $k$-linear combination $p_\Omega$ of a $(h_x+h_y+2h_z)$-homogeneous set of equations defining $Y$, and a general $k$-linear combination $p_\Theta$ of a $(h_x+2h_y+2h_z)$-homogeneous set of equations defining $X$. Let $\Theta=\mathbb{V}(p_\Theta)$ and $\Omega=\mathbb{V}(p_\Omega)$ and define $Z=\Theta \cup \Omega$. We compute $e_XZ$ using these polynomials and the generators of $\mathcal{I}_X$ above in conjunction with Theorem \ref{thm:eX_Y_MultiProj}. In this case $e_XZ=2>1$, hence $X\subset Y$. To obtain the integer $e_XZ$ we must solve one zero-dimensional system defined by multidegree $h_x+2h_y+2h_z$ polynomials. 

Using classical methods, we would verify the ideal containment $\mathcal{I}_Y:B^\infty\subset \mathcal{I}_X:B^\infty$ using Gr\"obner bases. Note that in this case $\mathcal{I}_Y \not\subset \mathcal{I}_X$, and hence to get the correct answer regarding containment as subvarieties of $\pp_x^2\times\pp_y^2\times\pp_z^2$ we \textit{must} compute the $B$-saturated ideals. For this example computing $\mathcal{I}_Y:B^\infty\subset \mathcal{I}_X:B^\infty$ using Gr\"obner bases takes approximately $11.4$ seconds, while computing $e_XZ=2$ takes approximately $2.3$ seconds (both using Macaulay2). In particular for this example, even though we use the more complicated ideal defining $Z$, the improvement comes from the fact that the method of Theorem \ref{thm:eX_Y_MultiProj} `knows' the geometric structure of the ambient space.
\end{example}\begin{wrapfigure}{r}{0.58\textwidth}
\vspace{-1.3em}
\begin{center}
\includegraphics[scale=0.13]{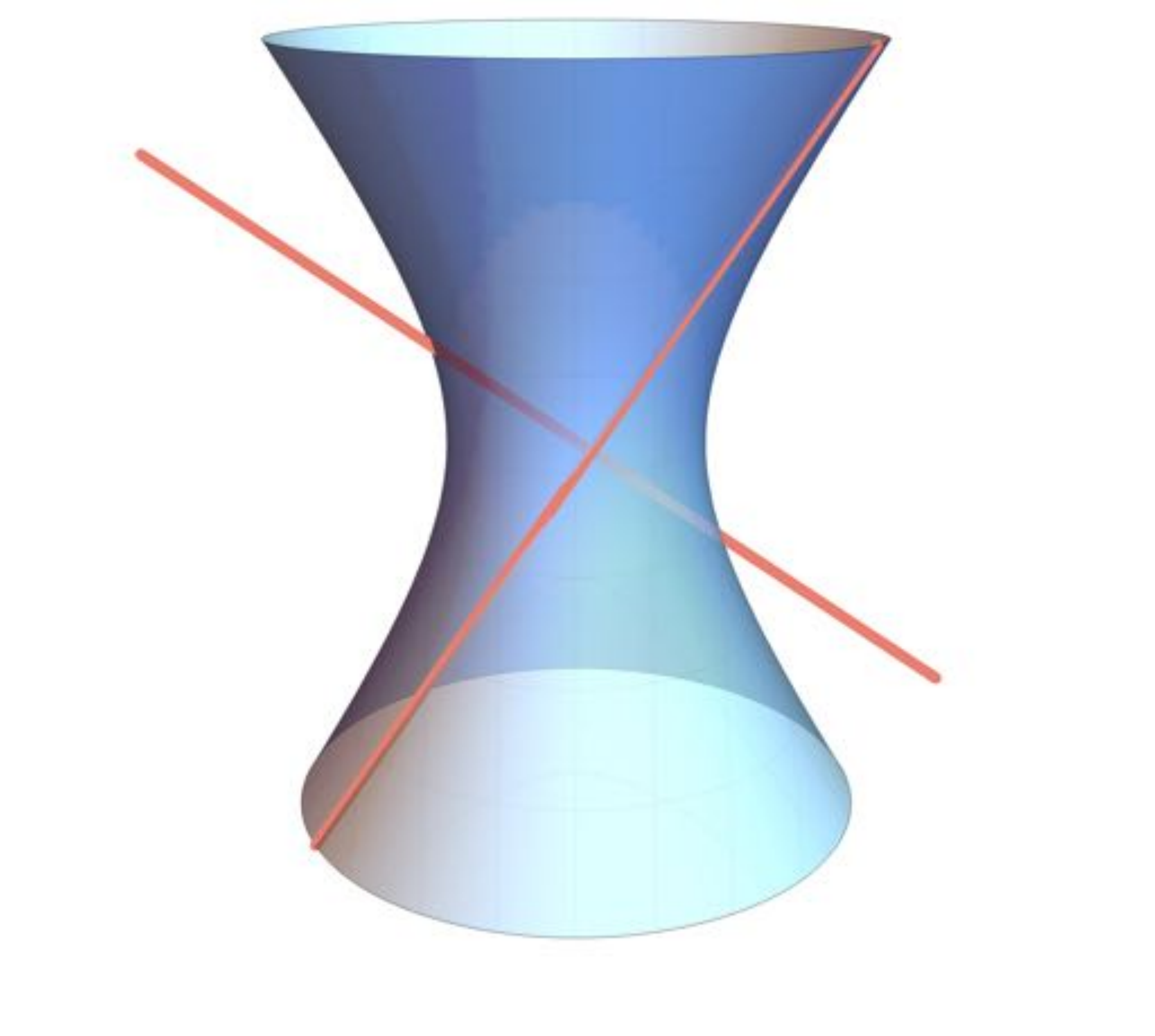}
\end{center}\caption{Two lines and a quadric surface.\label{fig:containment}}
\vspace{-2.5em}
 \end{wrapfigure} 

Now we consider the case where we are given an arbitrary pair of ideals.  What can the techniques of this section say about the containment of their associated algebraic sets?
In Figure \ref{fig:containment} we see a pair of lines, only one of which lies inside the quadric surface $Q$.  If $X = L_1 \cup L_2$ is the pair of lines, we cannot quite test whether $X \subset Q$.  Instead, as the next theorem shows, we can check if there exists $i$ such that $L_i \subset Q$.
\begin{theorem}
Let $X$ and $Y$ be arbitrary non-empty subschemes of $T_\Sigma$ with $Z = \Theta \cup \Omega$ defined as in Theorem \ref{thm:containmentGeneralIrrVar}.
Then a top-dimensional irreducible component $V$ of $X$ is contained in $Y$ if and only if 
\[
\Lambda_{a}(X,Z)\neq \Lambda_{a}(X,\Theta) 
\]
for some $a$ such that $|a|=\dim(X)$.
\label{thm:containmentGeneral}
\end{theorem}
\begin{proof}First suppose that $\Lambda_{a}(X,Z)= \Lambda_{a}(X,\Theta )$ for all $a$. This implies that $$\{s(X,Z)\}_{\dim(X)}=\{s(X,\Theta)\}_{\dim(X)}.$$ Suppose, for the sake of contradiction, that $ X_{\rm red}$ is contained in $Y$. Then it follows that $X_{\rm red}$ is also contained in $\Omega$. By Corollary \ref{cor:SegreDimXToric} we would then have that $$\{s(X,Z)\}_{\dim(X)}=\{s(X,\Omega)\}_{\dim(X)}+\{s(X,\Theta)\}_{\dim(X)},$$ but if $X_{\rm red}\subset \Omega$ and $X$ is non-empty then $ \{s(X,\Omega)\}_{\dim(X)}\neq 0$, which is a contradiction since $\{s(X,Z)\}_{\dim(X)}=\{s(X,\Theta)\}_{\dim(X)}$. Hence $X_{\rm red}$ is not contained in $Y$.

Now suppose that there exists an $a$ with $|a|=\dim(X)$ such that $\Lambda_{a}(X,Z)\neq \Lambda_{a}(X,\Theta) $. This implies that $\{s(X,Z)\}_{\dim(X)}\neq \{s(X,\Theta)\}_{\dim(X)}$. 
Hence by Corollary \ref{cor:SegreDimXToric} there must exist some irreducible component $V$ of $X$ such that $\dim(V)=\dim(X)$ and $V_{\rm red}$ is contained in some irreducible component of $Z$ other than $\Theta$.
So $V \subset\Omega$ and the proof proceeds as in the previous theorem.
\end{proof}
In the example below we demonstrate how the result of Theorem \ref{thm:containmentGeneral} can be applied to test if two different ideals have the same radical, without computing the radical. 
\begin{example}[Equality of radical ideals]
Work in $\pp^6$ with coordinate ring $R=k[x_0,\dots,x_6]$ and Chow ring $A^*(\pp^6)=\ZZ[h]/\langle h^7\rangle$. Let $f_1=x_2x_3x_5-5x_6^2x_0+3x_2x_0x_1$ and let $ f_2$ be a general homogeneous polynomial of degree three in $R$. Consider the variety $X$ defined by the ideal $\mathcal{I}_X=\langle f_1,f_2 \rangle$, and the irreducible scheme $Y$ defined by the ideal $\mathcal{I}_Y=\langle f_1^2,f_1f_2,f_2^2 \rangle$. We wish to verify that ${\mathcal{I}_X}=\sqrt{\mathcal{I}_Y}$, or equivalently that $X=Y_{\rm red}$ using Theorem \ref{thm:containmentGeneral}. Let $Z=\mathbb{V} (\lambda_1f_1^2+\lambda_2 f_1f_2+\lambda_3 f_2^2)\cup \mathbb{V}(\lambda_4f_1+\lambda_5 f_2) = \Omega \cup \Theta$ for general $\lambda_i\in k$. Using Theorem \ref{MainTheorem1_Multi_proj} we compute that
\[
\Lambda_{\dim(X)}(X,Z)=27 \quad \text{and}\quad \Lambda_{\dim(X)}(X,\Theta)=9.
\]
Hence by Theorem \ref{thm:containmentGeneral} we have that $X\subset Y$. 
Similarly
\[
\Lambda_{\dim(Y)}(Y,Z)=54 \quad \text{and} \quad \Lambda_{\dim(Y)}(Y,\Omega)=36,
\]
so $Y_{\rm red}\subset X$. It follows that $X=Y_{\rm red}$, and hence, that ${\mathcal{I}_X}=\sqrt{\mathcal{I}_Y}$. Note that we have implicitly used the fact that $X$ and $Y$ are irreducible (by construction). If we had only known a priori that $X$ was a variety, but had no knowledge of the structure of $Y$, we could have only concluded that an irreducible component of $Y$ was equal to $X$, or equivalently that $\mathcal{I}_X$ was a component of a primary decomposition of $\sqrt{\mathcal{I}_Y}$. 

The computation of all the values above using Theorem \ref{MainTheorem1_Multi_proj} takes approximately 2.4 seconds (in Macaulay2 \cite{M2}) and is performed using the ideals $\mathcal{I}_X$ and $\mathcal{I}_Y$ directly; that is we \textit{do not} compute the radical of $\mathcal{I}_Y$. On the other hand we could also verify that ${\mathcal{I}_X}=\sqrt{\mathcal{I}_Y}$ by computing these radicals directly and using Gr\"obner bases to check ideal equality. 
We stopped this computation after 1.5 hours.
In short, we learned geometric information about the varieties associated to ideals without computing radicals.
\end{example}

Finally we give a criterion using the techniques of \S\ref{sec:MainResults} for testing whether an ideal defines the empty set.
\begin{theorem}
Let $B$ be a possibly-empty subscheme of $\pp^n$ defined by the ideal $\mathcal{I}_B=\langle f_0,\dots, f_r\rangle$ and let $d=\max_i(\deg(f_i))$.
If $B = \emptyset$, then $g_0(B,\pp^n) = d^{n}$.
Moreover, $B = \emptyset$ if and only if
the projective degrees $g_i(B,\pp^n)=d^{n-i}$ for all $i$.
\end{theorem}
\begin{proof}
This follows immediately from the remark made before \eqref{eq:Lambda}.
\end{proof}
\begin{small}

\noindent
{\bf Acknowledgements.}
Martin was partially supported by the Independent Research Fund of Denmark during the preperation of this work.
Corey was partially supported by the Bergen Research Foundation project grant ``Algebraic and topological cycles in tropical and complex geometry.''
\end{small}
\bibliographystyle{alpha}
\bibliography{library}

\begin{thebibliography}{HDSV18}

\bibitem[AH17]{AH17}
Paolo Aluffi and Corey Harris.
\newblock The {E}uclidean distance degree of smooth complex projective
  varieties.
\newblock {\em arXiv preprint arXiv:1708.00024}, 2017.

\bibitem[Alu03]{Aluffi2003c}
Paolo Aluffi.
\newblock {Computing characteristic classes of projective schemes}.
\newblock {\em Journal of Symbolic Computation}, 35(1):3--19, 2003.

\bibitem[Alu18]{aluffi2018chern}
Paolo Aluffi.
\newblock The {C}hern-{S}chwartz-{M}ac{P}herson class of an embeddable scheme.
\newblock {\em arXiv preprint arXiv:1805.11116}, 2018.

\bibitem[BHSW]{BHSW06}
Daniel~J. Bates, Jonathan~D. Hauenstein, Andrew~J. Sommese, and Charles~W.
  Wampler.
\newblock Bertini: Software for numerical algebraic geometry.
\newblock Available at bertini.nd.edu with permanent doi:
  dx.doi.org/10.7274/R0H41PB5.

\bibitem[CLS11]{cox2011toric}
David~A. Cox, John~B. Little, and Henry~K. Schenck.
\newblock {\em Toric varieties}, volume 124 of {\em Graduate Studies in
  Mathematics}.
\newblock American Mathematical Society, Providence, RI, 2011.

\bibitem[Dan78]{Danilov1978}
V~I Danilov.
\newblock {the Geometry of Toric Varieties}.
\newblock {\em Russian Mathematical Surveys}, 33(2):97--154, 1978.

\bibitem[EJP13]{Eklund2013}
David Eklund, Christine Jost, and Chris Peterson.
\newblock {A method to compute {S}egre classes of subschemes of projective
  space}.
\newblock {\em J. Algebra Appl.}, 12(2):15,1250142, 2013.

\bibitem[Ful98]{fulton2013intersection}
William Fulton.
\newblock {\em Intersection theory}.
\newblock Springer-Verlag New York, second edition, 1998.

\bibitem[GS]{M2}
Daniel~R Grayson and Michael~E Stillman.
\newblock {Macaulay2, a software system for research in algebraic geometry}.
\newblock Available at \url{http://www.math.uiuc.edu/Macaulay2}.

\bibitem[Har92]{harris1992algebraic}
Joe Harris.
\newblock {\em Algebraic geometry: a first course}, volume 133.
\newblock Springer Science \& Business Media, 1992.

\bibitem[Har17]{Harris2017}
Corey Harris.
\newblock {Computing Segre classes in arbitrary projective varieties}.
\newblock {\em Journal of Symbolic Computation}, 82:26--37, sep 2017.

\bibitem[HDSV18]{hauenstein2018solving}
Jonathan~D Hauenstein, Mohab Safey~El Din, {\'E}ric Schost, and Thi~Xuan Vu.
\newblock Solving determinantal systems using homotopy techniques.
\newblock {\em arXiv preprint arXiv:1802.10409}, 2018.

\bibitem[Hel16]{helmer2016proj}
Martin Helmer.
\newblock Algorithms to compute the topological {E}uler characteristic,
  {C}hern--{S}chwartz--{M}ac{P}herson class and {S}egre class of projective
  varieties.
\newblock {\em Journal of Symbolic Computation}, 73:120--138, 2016.

\bibitem[Hel17]{helmer2017toric}
Martin Helmer.
\newblock Computing characteristic classes of subschemes of smooth toric
  varieties.
\newblock {\em Journal of Algebra}, 476:548--582, 2017.

\bibitem[Kle74]{kleiman1974transversality}
Steven~L Kleiman.
\newblock The transversality of a general translate.
\newblock {\em Compositio Math}, 28:287--297, 1974.

\bibitem[Lec]{Kronecker}
G.~Lecerf.
\newblock Kronecker: Polynomial equation system solver.
\newblock Software Available at
  \url{http://www.lix.polytechnique.fr/~lecerf/software/kronecker/index.html};
  \url{http://www.mathemagix.org/www/geomsolvex/doc/html/index.en.html}.

\bibitem[Lec03]{lecerf2003computing}
Gr{\'e}goire Lecerf.
\newblock Computing the equidimensional decomposition of an algebraic closed
  set by means of lifting fibers.
\newblock {\em Journal of Complexity}, 19(4):564--596, 2003.

\bibitem[MQ13]{Moe2013}
Torgunn~Karoline Moe and Nikolay Qviller.
\newblock {Segre classes on smooth projective toric varieties}.
\newblock {\em Mathematische Zeitschrift}, 275(1-2):529--548, 2013.

\bibitem[MS04]{miller2004combinatorial}
Ezra Miller and Bernd Sturmfels.
\newblock {\em Combinatorial commutative algebra}, volume 227.
\newblock Springer Science \& Business Media, 2004.

\bibitem[Pie78]{Piene1978}
Ragni Piene.
\newblock {Polar classes of singular varieties}.
\newblock 2:247--276, 1978.

\bibitem[Sam55]{samuel1955methodes}
Pierre Samuel.
\newblock {\em M{\'e}thodes d'alg{\`e}bre abstraite en g{\'e}om{\'e}trie
  alg{\'e}brique}.
\newblock Ergebnisse der Math., Springer-Verlag, 1955.

\bibitem[Say17]{sayrafi2017computations}
Mahrud Sayrafi.
\newblock Computations over local rings in {M}acaulay2.
\newblock {\em arXiv preprint arXiv:1710.09830}, 2017.

\bibitem[Ver99]{verschelde1999algorithm}
Jan Verschelde.
\newblock Algorithm 795: {PHC}pack: A general-purpose solver for polynomial
  systems by homotopy continuation.
\newblock {\em ACM Transactions on Math.~Software (TOMS)}, 25(2):251--276,
  1999.

\end{thebibliography}

\medskip
\medskip

\noindent
 {\bf Authors' addresses:}
\small

\smallskip

\noindent\\
{\bf Corey Harris}: \\
Department of Mathematics, University of Oslo\\
P.O.~Box 1053 Blindern, 0316 Oslo, Norway\\
{\em Email address}: {\tt charris@math.uio.no}\\

\smallskip

\noindent 
{\bf Martin Helmer} (\textit{Corresponding Author}): \\
Department of Mathematical Sciences, 
University of Copenhagen \\
Universitetsparken 5, 
DK-2100 Copenhagen, Denmark\\
{\em Email address}: {\tt martin.helmer2@gmail.com}\\
\end{document}